\documentclass{scrartcl}
\usepackage{graphicx} 
\usepackage{microtype}
\usepackage[section]{algorithm}
\usepackage{algorithmic}
\usepackage{hyperref}
\usepackage{url}

\usepackage{amsmath}
\usepackage{amsthm}
\usepackage{amsfonts}
\usepackage{amssymb}
\usepackage{mathrsfs}

\usepackage[usenames,dvipsnames]{xcolor}

\usepackage{pgf}
\usepackage{pgfplots}
\pgfplotsset{compat=newest}

\usepackage{booktabs}
\usepackage{multirow}
\usepackage{dcolumn}
\usepackage{colortbl}
\usepackage[hang,small,bf]{subfigure}

\usepackage{pmat}
\usepackage{bbm}

\usepackage{bm}

\usepackage{verbatim} 

\usepackage{tikz}
\usepackage{marvosym}

\usepackage{empheq}

\usepackage{blkarray} 
\usepackage{arydshln} 
\usepackage{setspace}

\usepackage{textcomp}

\usepackage{cite}

\numberwithin{equation}{section}

\newtheorem{theorem}{Theorem}[section]
\newtheorem{definition}[theorem]{Definition}
\newtheorem{lemma}[theorem]{Lemma}

\newtheorem{remark}[theorem]{Remark}
\newtheorem{example}{Example}[section]

\newlength\plotheight
\newlength\plotwidth

\definecolor{pureblack}{gray}{0}
\arrayrulecolor{pureblack}

\newcommand{\RIND}[1]{\mbox{$\scriptstyle{#1}$}}

\allowdisplaybreaks

\makeatletter
\DeclareOldFontCommand{\rm}{\normalfont\rmfamily}{\mathrm}
\DeclareOldFontCommand{\sf}{\normalfont\sffamily}{\mathsf}
\DeclareOldFontCommand{\tt}{\normalfont\ttfamily}{\mathtt}
\DeclareOldFontCommand{\bf}{\normalfont\bfseries}{\mathbf}
\DeclareOldFontCommand{\it}{\normalfont\itshape}{\mathit}
\DeclareOldFontCommand{\sl}{\normalfont\slshape}{\@nomath\sl}
\DeclareOldFontCommand{\sc}{\normalfont\scshape}{\@nomath\sc}
\makeatother

\begin{document}
\title{iSIRA: Integrated Shift-Invert Residual Arnoldi Method for Graph Laplacian Matrices from Big Data}%
\author{Wei-Qiang Huang\thanks{Big Data Research Center, National Chiao Tung University, Hsinchu 300, Taiwan, %
\href{mailto:wqhuang@nctu.edu.tw}{wqhuang@nctu.edu.tw}.} %
\and Wen-Wei Lin\thanks{Department of Applied Mathematics and Shing-Tung Yau Center, National Chiao Tung University, Hsinchu 300, Taiwan, %
\href{mailto:wwlin@math.nctu.edu.tw}{wwlin@math.nctu.edu.tw}.} %
\and Henry Horng-Shing Lu\thanks{Institute of Statistics and Big Data Research Center, National Chiao Tung University, Hsinchu 300, Taiwan, %
\href{mailto:gu@cs.sunysb.edu}{hslu@stat.nctu.edu.tw}.} %
\and Shing-Tung Yau\thanks{Mathematics Department, Harvard University, Cambridge, Massachusetts, 02138, USA, %
\href{mailto:yau@math.harvard.edu}{yau@math.harvard.edu}.}}
\maketitle

\begin{abstract}
The eigenvalue problem of a graph Laplacian matrix $L$ arising from a simple, connected and undirected graph has been given more attention %
due to its extensive applications, such as spectral clustering, community detection, complex network, image processing and so on. %
The associated graph Laplacian matrix is symmetric, %
positive semi-definite, and is usually large and sparse. %
Computing some smallest positive eigenvalues and corresponding eigenvectors is often of interest. %

However, the singularity of $L$ makes the classical eigensolvers inefficient since we need to factorize $L$ for the purpose of solving large and sparse linear systems exactly. %
The next difficulty is that it is usually time consuming or even unavailable to %
factorize a large and sparse matrix arising from real network problems from big data %
such as social media transactional databases, and sensor systems because there is in general not only local connections.%

In this paper, we propose an eignsolver based on the inexact residual Arnoldi \cite{L07,LS07} method together with an %
{\it implicit remedy of the singularity} and an {\it effective deflation for convergent eigenvalues}. %
Numerical experiments reveal that the integrated eigensolver outperforms the classical Arnoldi/Lanczos method for computing some smallest positive eigeninformation provided the {\normalfont LU} factorization is not available.%
\end{abstract}

\newcommand{\keywords}[1]{\paragraph{Keywords:} #1}

\newcommand{\subclass}[1]{\paragraph{Mathematics Subject Classification (2010):} #1}
\subclass{15B99 \and 65F50 \and 65N25}

\section{Introduction}\label{sec:Introduction}
Graph-based approaches have been an increasingly favorable tool for  representation, processing and explosion of interests in studying large networks %
because of the potential for capturing dependence structure of the data set. %
One of the research problems of interests is to compute some eigenpairs based on the so-called Graph Laplacians Eigenvalue Problem (GLEP). %
The GLEPs appear in many areas, such as combinatorial optimization, clustering, embedding, dimensionality reduction, data representation, %
community detection, image processiong and complex networks 
\cite{MP93,SM00,NJY02,BN03,XW05,WS05,N06,vL07,S07,ADGKMZ08,F10,HGLY14,CH15IP,CH15A,PCOA15}.
For more applications on this topic, we refer to as \cite{M97}. %

\subsection{The Graph Laplacian Matrix}\label{sec:laplacian}

Given a simple (no multiple edges or loops), connected and undirected graph $\mathcal{G}=\left(\mathcal{V},\mathcal{E}\right)$, %
where $\mathcal{V}$ is the vertex set with $| \mathcal{V} | = n$ and $\mathcal{E}$ is the edge set that describes the connection between vertices. %
In addition, the degree of each vertex $v \in \mathcal{V}$ is the number of edges which incident to $v$. %
According to the information, the adjacency matrix $A$ as well as the degree matrix $D$ of $\mathcal{G}$ are, respectively, defined by%
\begin{equation*}
A_{ij} = \left\{
                   \begin{array}{cc}
                    1, & \text{if }ij \in    \mathcal{E}, \\
                    0, & \text{if }ij \notin \mathcal{E},%
                   \end{array}%
\right.\quad \text{and}\quad %
D_{ij} = \left\{
                   \begin{array}{cc}
                    \sum_{j=1}^{n}A_{ij}, & \text{if }i=j,    \\
                    0                   , & \text{if }i\neq j,%
                   \end{array}
                   \right.
\end{equation*}
where $1\leq i,j\leq n$. %
Then the so-called graph Laplacian matrix of $\mathcal{G}$ is defined by $L = D - A$. That is,
\begin{equation}\label{eqn:laplacian}
    L_{ij} = \left\{
                       \begin{array}{cl}
                            D_{ii}, & \text{if }i=j, \\
                            -1, & \text{if }i\text{ and }j\text{ are adjacent,} \\
                            0, & \text{otherwise,}%
                       \end{array}%
                       \right.%
    \ 1\leq i,j\leq n.
\end{equation}
Below, we recall some properties of the graph Laplacian matrix.%
\begin{definition}
  Recall the definitions of the irreducible matrix and the {\normalfont M}-matrix.
  \begin{itemize}
    \item[(i)] A matrix is called an irreducible matrix if it is not similar to a block upper triangular matrix using a permutation matrix.
    \item[(ii)] A square matrix is said to be an {\normalfont M}-matrix if its off-diagonal entries are less than or equal to zero with eigenvalues whose real parts are positive.
  \end{itemize}
\end{definition}

\begin{theorem}[\cite{M12}]\label{thm:laplacian}
Let $L$ be a graph Laplacain matrix induced from a simple, connected and undirected graph. Then
\begin{itemize}
\item[(i)] $L$ is a singular irreducible {\normalfont M}-matrix.

\item[(ii)] $L$ has exactly one zero eigenvalue and the corresponding eigenvector is the all-one vector.
\end{itemize}
\end{theorem}

\begin{remark}
If we additionally associate $\mathcal{G}$ with a symmetric nonnegative weight matrix $W \in \mathbb{R}^{n\times n}$ with zero diagonal entries. %
That is, the entry $W_{ij}$ denotes the nonnegative edge weight between vertices $i, j$ if $ij \in \mathcal{E}$ and %
$W_{ij} = 0$ in the case either $i = j$ or vertex $i$ as well as $j$ are not connected. %
At this time, the diagonal degree matrix $D_{w}$ is composed of the total weight on edges connected to each vertex %
and $L_{w} := D_{w} - W$ is called the weighted Laplacian $L_{w}$ associated to $W$. %
It is of course that weights on edges can be signed. This study, we only focus on positive weights. %
To simplify the notation, we only consider $W = A$, that is the case of the classical Laplacian. %
Note that the following discussion is also true for a graph Laplacian matrix with positive edge weights. %
\end{remark}

Finding some smallest positive eigenvalues and the associated eigenvectors of the graph Laplacian matrix \eqref{eqn:laplacian} is a fundamental problem in these applications.
Certainly, there have been many excellent theoretical investigations and numerical algorithms for these subjects %
and these methods, such as the inverse power method \cite{I97} and the Anoldi method as well as the Lanczos method \cite{GvL12}, 
Nevertheless, these methods have to continuously solve linear systems of the rank deficient Laplacian matrix %
generated from a simple, connected and undirected graph. %
A traditional way to address the singularity problem is to add a small diagonal perturbation matrix. %
The same trick will be applied to improve the condition number of the problem, %
through the Tikhonov regularization \cite{TA77} or the ridge regression \cite{N04} %
in the area of statistics and machine learning.
Even the resulting matrix is invertible, %
for large graph such as the social network, however, there is not only local connections between the vertices %
and this implies that it is almost impossible to find matrix factorizations %
for exact solving linear systems involved in an eigensolver. %
That is, only the most scalable algorithms are practical for a large graph Laplacian matrix. %

In this paper we first remedy the problem of singularity of the graph Laplacian matrix. %
Then, we introduce a technique of eigenvalue deflations as shown in \cite{P98,S01,GvL12,HGLY14} so that we can find the desired eigenvalues in order %
and exclude the influence of the convergent ones. %
Finally, we integrate these approaches into the eigensolver called shift-invert residual Arnoldi (SIRA) method \cite{L07,LS07}. %
SIRA is an inner-outer iterative eigensolver and the inner linear system is allowed to be solved with the low accuracy. %
As a result, we propose an inexact inner-outer eigensolver with the cure of sigularity and the aid of deflation.

\subsection{Notations and Overview}\label{subec1.2:notation}

Throughout this paper, capital Roman and Greek letters indicate matrices and lowercase bold face letters denote vectors. %
Lowercase Greek letters are the scalars. $I_{n}$ denotes the $n \times n$ identity matrix with the given size $n$, %
$\mathbf{e}_j$ is the $j$th column of the identity matrix $I_{n}$. The notation $\mathbf{1}_{n}$ denotes a all-one $n$-vector. %
$\mathbf{0}$ represents a zero vector and matrix whose dimension should become evident from the context. %
We adopt the following MATLAB notations:
$\mathbf{v}(i\text{:}j)$ denotes the subvector of the vector $\mathbf{v}$ that consists of the $i$th to the $j$th entries of $\mathbf{v}$. %
Given a matrix $A$, the slice $A(i\text{:}j,:)$ selects the rows $i$ to $j$ and $A(i,:)$ indicates the $i$th row of $A$. %
The notation $\cdot ^{\top}$ denotes the transpose of vectors or matrices. %
Other notations will be clearly defined whenever they are used. %

The rest of this paper is organized as follows. %
In Section~\ref{sec:GLEP}, we revisit the background on the eigenvalue problem of a graph Laplacian matrix $L$ in \eqref{eqn:laplacian}. %
Next, we propose trimming as well as deflating approaches, and review the shift-inverted residual Arnoldi method. %
In Section~\ref{sec:Implement}, %
we integrate the above techniques and present an iterative eigensolver for solving the large and spare graph Laplacain eigenvalue problems. %
Numerical experiments and comparisons are presented in Section~\ref{sec:Numerical}. %
Finally, we end up by discussing some concluding remarks and further works in Section~\ref{sec:Conclusion}.

\section{The Graph Laplacian Eigenvalue Problem}\label{sec:GLEP}

For the graph Laplacian matrix $L$ \eqref{eqn:laplacian} generated from a simple-connected graph with nodes $n$, we study the eigenvalue problem
\begin{equation}\label{eqn:GLEP}
    L\mathbf{x} = \lambda \mathbf{x},
\end{equation}
where $L$ is symmetric positive semi-definite with the unique kernel vector $\mathbf{1}_{n}$ as proposed in Theorem~\ref{thm:laplacian}. %
We are interested in finding some smallest positive eigenvalues and the associated eigenvectors of \eqref{eqn:GLEP}. %

It is well know that the shift-invert spectral transformation \cite{P98,B_00,S01,GvL12} is used to enhance convergence to a desired portion of the spectrum. %
Specifically, to find eigenvalues of $L\mathbf{x} = \lambda \mathbf{x}$ near a target $\sigma$, which is not an eigenvalue of $L$, %
the shift-invert spectral transformation is to consider the corresponding shift-and-invert eigenvalue problem $(L-\sigma I_{n})^{-1}\mathbf{x} = \theta \mathbf{x}$. %
Once we find an aforesaid eigenvalue $\theta$, it can then be transformed back to eigenvalues of the original problem. %
The direct relation is $\lambda = \sigma + 1/\theta$. However, we cannot just apply this technique to the graph Laplacian matrix. %

First note that since $L$ is positive semi-definite, we know that its eigenvalue with the smallest magnitude is $\lambda_0 = 0$.
Secondly, in order to finding some smallest positive eigenvalues of the graph Laplacian matrix $L$, %
the zero-shift $\sigma = 0$, without any additional information, would be an intuitive and appropriate selection. %
This means that we would make use of the zero-shift and consider the invert problem of the singular matrix $L$. %
In this case, we will encounter the problem of solving singular linear systems $L\mathbf{z} = \mathbf{r}$. %
To remedy this defect, we first propose a trimming technique in Section~\ref{sec:trimming}. %

Secondly, from the fact that the convergent ones will influence the convergence process of other eigenvalues that have not yet been captured, %
Section~\ref{sec:deflting} introduces the method of deflation to exclude the impact of convergence from convergent eigenvalues. %

Lastly, although we can always focus on finding the smallest eigenvalue after deflating these convergent eigenvalues, %
shift-invert method still have to be included. %
This is because the next desired eigenvalue is getting farther and farther away from the origin. %
To integrate these techniques, trimming, deflating, and shift-invert is not obvious. The details will be explained in Section~\ref{sec:sidevp}.

\subsection{Trimming Technique}\label{sec:trimming}

To use the invert technique for computing the smallest positive eigenvalues of $L$ is equivalent to find the largest positive one of its inverted. %
Therefore, we have to solve the linear system $L\mathbf{z} = \mathbf{r}$. %
Besides, the vector $\mathbf{z}$ is usual used to expand a search subspace for a project eigensolver (such as the Lanczos method), %
we can further require that $\mathbf{z}$ and $\mathbf{1}_{n}$ are orthogonal to each other. %
If $L\mathbf{z} = \mathbf{r}$ has a solution $\mathbf{z}$ that is perpendicular to $\mathbf{1}_{n}$, then the following result will be obtained.%
\begin{lemma}\label{lem:rortho1}
    The orthogonality requirement of $\mathbf{r}$ and $\mathbf{1}_{n}$  is a necessary condition provided the constrained linear system, %
    \begin{equation*}
      L\mathbf{z} = \mathbf{r} \quad \text{with} \quad \mathbf{1}_{n}^{\top}\mathbf{z} = 0,\ \text{where}\ L^{\top} = L, \ L\mathbf{1}_{n} = \mathbf{0},
    \end{equation*}
    has a solution.
\end{lemma}
\begin{proof}
     If the system is consistency, we have $\mathbf{r}^{\top}\mathbf{1}_{n} = \mathbf{z}^{\top}L\mathbf{1}_{n} = \mathbf{z}^{\top}\mathbf{0} = 0$. %
\end{proof}
Overall, we will solve a constrained singular linear system
\begin{equation}\label{eqn:sinsysm}
\left\{
    \begin{array}{l}
        L\mathbf{z} = \mathbf{r} \quad \text{with} \quad \mathbf{1}_{n}^{\top}\mathbf{z} = 0,      \smallskip\\
        \text{where } L^{\top} = L, \ L\mathbf{1}_{n} = \mathbf{0}\ \text{and}\ \mathbf{1}_{n}^{\top}\mathbf{r} = 0.%
    \end{array}%
\right.
\end{equation}
Fortunately, as known in \cite[Theorem~4.31]{Q09}, each principal submatrix of $L$ with order less than $n$ is a nonsingular {\normalfont M}-matrix. %
\begin{theorem}[{\cite[Theorem~4.31~(d)]{Q09}}]\label{thm:invMmat}
  Let $L$ be the graph Laplacian matrix constructed from a simple, connected and undirected graph. Write %
  \begin{equation}\label{eqn:Lhati}
  L =\begin{blockarray}{c@{}ccc@{}cl}
                &                         & {\scriptsize \rotatebox{90}{\begin{tikzpicture}\node (scissors) {\Leftscissors};\end{tikzpicture}}} & &             \\
            \begin{block}{[c@{\hspace{5pt}}ccc@{\hspace{5pt}}c]l}
                & \widehat{L}_{i,1}       & |       & \widehat{L}_{i,2} & &                                                                                     \\
                & \text{---}              & \l_{ii} & \text{---}        & &{\scriptsize \begin{tikzpicture}\node (scissors) {\Leftscissors};\end{tikzpicture}}  \\
                & \widehat{L}_{i,2}^\top  & |       & \widehat{L}_{i,3} & &                                                                                     \\
            \end{block}
         \end{blockarray}
    \quad \text{and} \qquad \widehat{L} =
    \begin{bmatrix}
       \widehat{L}_{i,1}      & \widehat{L}_{i,2} \\
       \widehat{L}_{i,2}^\top & \widehat{L}_{i,3}
    \end{bmatrix}.
  \end{equation}
That is, $\widehat{L}$ represents the submatrix of $L$ with order $n-1$ obtained by removing the $i$th row and column. %
Then the $\widehat{L}$ is a nonsingular {\normalfont M}-matrix.
\end{theorem}

\begin{remark}
The index $i$ can be any integer number between $1$ and $n$, and we do not specifically declare $i$. %
Specified index number should, if necessary, be evident from the context.
\end{remark}

\begin{proof}[Proof of Theorem~\ref{thm:invMmat}]
  For the special case that $L$ is a graph Laplacian matrix, we give a compact proof. %
  Clearly, $\widehat{L}$ is a symmetric matrix with nonpositive off-diagonal entries, we next claim that $\widehat{L}$ is positive definite. %
  If not, we then have a nonzero vector $\widehat{\mathbf{x}}\in\mathbb{R}^{n-1}$ %
  such that $\widehat{\mathbf{x}}^\top \widehat{L}\widehat{\mathbf{x}} \leq 0$.
  However, if we consider the corresponding enlarged vector 
  $\mathbf{x}=\big[\widehat{\mathbf{x}}(1:i-1)^\top\ \ 0\ \ \widehat{\mathbf{x}}(i:n-1)^\top\big]^\top$, %
  we will get $\mathbf{x}^\top L\mathbf{x} \leq 0$, which contradicts to the fact $\mathbf{1}_{n}$ %
  is the only one kernel vector of the positive semi-definite matrix $L$. Therefor, we show that $\widehat{L}$ is a nonsingular M-matrix.
\end{proof}

Based on this theorem, we can first solve the an $(n-1) \times (n-1)$ linear system with the trimmed coefficient matrix $\widehat{L}$ in \eqref{eqn:Lhati}, %
and the corresponding right-hand side that removes the $i$th element. After that, it is enlarged to an $n$-vector as a solution of \eqref{eqn:sinsysm}. %
In fact, we have the following theorem. %

\begin{theorem}\label{thm:trimming}
    Let $\widehat{L}$ be the matrix defined in \eqref{eqn:Lhati} and %
    $\widehat{\mathbf{r}}$ be the $(n-1)$-vector obtained from $\widehat{\mathbf{r}}$ by deleting the same $i$th element. %
    If $\widehat{\mathbf{z}}^{\ast}$ is the solution of %
    \begin{equation}\label{eqn:tlinsys}
        \widehat{L}\widehat{\mathbf{z}}=\widehat{\mathbf{r}}.
    \end{equation}
    Then the enlarged $n$-vector %
    \begin{equation}\label{eqn:z_i}\mathbf{z}^{\ast} :=
        \begin{bmatrix}
            \widehat{\mathbf{z}}^{\ast}(1\text{\normalfont{:}}i-1) \\ %
            0 \\ %
            \widehat{\mathbf{z}}^{\ast}\left(i\text{\normalfont{:}}n-1\right)
        \end{bmatrix}
        - \frac{\widehat{\mathbf{1}}_{n}^{\top}\widehat{\mathbf{z}}^{\ast}}{n}\mathbf{1}_{n}
    \end{equation}
    is a solution of \eqref{eqn:sinsysm} that is orthogonal to $\mathbf{1}_{n}$. %
    Here, for consistency of representation, $\widehat{\mathbf{1}}_{n}$ is the all-ones vector with size $n-1$, that is, $\widehat{\mathbf{1}}_{n} = \mathbf{1}_{n-1}$. %
\end{theorem}

\begin{proof}
    For simplicity, we consider the leading principal submatrix of $L$, which is obtained by deleting the last row and column of $L$. %
    Thus, we can have the following partitions:%
    \begin{equation}\label{eqn:Lhzhrh}
        L = \begin{blockarray}{c@{}cc@{\hspace{4pt}}cl}
                & \RIND{n-1}                 & \RIND{1}            & &            \\
            \begin{block}{[c@{\hspace{5pt}}cc@{\hspace{5pt}}c]l}
                & \widehat{L}                & \widehat{\bm{\ell}} & & \RIND{n-1} \\
                & \widehat{\bm{\ell}}^{\top} & l                   & & \RIND{1}   \\
            \end{block}
         \end{blockarray},\quad
         \mathbf{z} = \begin{blockarray}{c@{}c@{\hspace{4pt}}cl}
                            &                      & & \\
                        \begin{block}{[c@{\hspace{5pt}}c@{\hspace{5pt}}c]l}
                            & \widehat{\mathbf{z}} & & \RIND{n-1} \\
                            & \zeta                & & \RIND{1}   \\
                        \end{block}
                      \end{blockarray}
         \quad \text{and} \quad
         \mathbf{r} = \begin{blockarray}{c@{}c@{\hspace{4pt}}cl}
                            &                      & & \\
                        \begin{block}{[c@{\hspace{5pt}}c@{\hspace{5pt}}c]l}
                            & \widehat{\mathbf{r}} & & \RIND{n-1} \\
                            & \rho                 & & \RIND{1}   \\
                        \end{block}
                      \end{blockarray}.
    \end{equation}%
    Then, solving \eqref{eqn:sinsysm} is equivalent to finding an $n$-vector $\mathbf{z} = \Big[ { \widehat{\mathbf{z}} \atop \zeta } \Big]$ satisfying %
    \begin{equation}\label{eqn:hateqn}
        \left\{
            \begin{array}{l}
                \widehat{L}\widehat{\mathbf{z}}               + \widehat{\bm{\ell}}\zeta = \widehat{\mathbf{r}}, \\
                \widehat{\bm{\ell}}^\top\widehat{\mathbf{z}}  + l\zeta                   = \rho\ ( = - \widehat{\mathbf{1}}_{n}^\top \widehat{\mathbf{r}} ), \\
                \zeta = - \widehat{\mathbf{1}}_{n}^\top \widehat{\mathbf{z}}.
            \end{array}
        \right.
    \end{equation}

    As mentioned in Theorem~\ref{thm:invMmat}, we know that $\widehat{L}$ is symmetric positive definite. %
    Notice that from Theorem~\ref{thm:laplacian} $(L\mathbf{1}_{n} = \mathbf{0})$ and the notations introduced in \eqref{eqn:Lhzhrh}, %
    we have $\widehat{\bm{\ell}} = -\widehat{L}\widehat{\mathbf{1}}_{n}$ and $l = - \widehat{\mathbf{1}}_{n}^{\top}\widehat{\bm{\ell}} = \widehat{\mathbf{1}}_{n}^{\top}\widehat{L}\widehat{\mathbf{1}}_{n}$. %
    Hence, \eqref{eqn:sinsysm} (or equivalently, \eqref{eqn:hateqn}) can be simplified by solving the $(n-1) \times (n-1)$ linear system %
    \begin{equation}\label{eqn:timmedlinsym0}
        \widehat{L} \Big(\widehat{I}_n + \widehat{\mathbf{1}}_{n}\widehat{\mathbf{1}}_{n}^{\top}\Big)\widehat{\mathbf{z}} = \widehat{\mathbf{r}}, %
        \quad \text{where}\quad \widehat{I}_n := I_{n-1}.
    \end{equation}
    Let $\widehat{\mathbf{z}}^{\ast} \in \mathbb{R}^{n-1}$ be the solution of $\widehat{L}\widehat{\mathbf{z}} = \widehat{\mathbf{r}}$, %
    then, from the observation %
    \begin{equation}\label{eqn:rank1inv}
        \left(\widehat{I}_n + \widehat{\mathbf{1}}_{n}\widehat{\mathbf{1}}_{n}^{\top}\right) %
        \left(\widehat{I}_n - \tfrac{1}{n}\widehat{\mathbf{1}}_{n}\widehat{\mathbf{1}}_{n}^{\top}\right) = \widehat{I}_{n}\ \Leftrightarrow\ %
        \left(\widehat{I}_n + \widehat{\mathbf{1}}_{n}\widehat{\mathbf{1}}_{n}^{\top}\right)^{-1} = %
              \widehat{I}_n - \tfrac{1}{n}\widehat{\mathbf{1}}_{n}\widehat{\mathbf{1}}_{n}^{\top},
    \end{equation}
    we know that $\widehat{\mathbf{z}}^{\ast} - \tfrac{\widehat{\mathbf{1}}_{n}^{\top}\widehat{\mathbf{z}}^{\ast}}{n}\widehat{\mathbf{1}}_{n}$ is the solution of %
    \eqref{eqn:timmedlinsym0}.
    Next, we use $\widehat{\mathbf{z}}^{\ast}$ to construct an $n$-vector $\mathbf{z}^{\ast}$ defined by %
    \begin{equation}\label{eqn:z}
        \mathbf{z}^{\ast} :=
            \begin{bmatrix} \widehat{\mathbf{z}}^{\ast} \\ 0 \end{bmatrix} - \zeta^{\ast}\mathbf{1}_{n}  \quad \text{with} \quad%
            \zeta^{\ast} = \frac{\widehat{\mathbf{1}}_{n}^{\top}\widehat{\mathbf{z}}^{\ast}}{n},
    \end{equation}
    then, noting that $\widehat{\bm{\ell}} = -\widehat{L}\widehat{\mathbf{1}}_{n}$ and $l = - \widehat{\mathbf{1}}_{n}^{\top}\widehat{\bm{\ell}} = \widehat{\mathbf{1}}_{n}^{\top}\widehat{L}\widehat{\mathbf{1}}_{n}$ , we can deduce that %
    \begin{itemize}
        \item[(i)] $\widehat{L}(\widehat{\mathbf{z}}^{\ast} - \zeta^{\ast}\widehat{\mathbf{1}}_{n}) %
                        + \widehat{\bm{\ell}}(-\zeta^{\ast}) %
                    = \widehat{L}\widehat{\mathbf{z}}^{\ast} = \widehat{\mathbf{r}}$;
        \item[(ii)] $\widehat{\bm{\ell}}^\top(\widehat{\mathbf{z}}^{\ast} - \zeta^{\ast}\widehat{\mathbf{1}}_{n}) %
                     + l(-\zeta^{\ast}) %
                     = - \widehat{\mathbf{1}}_{n}^{\top}\widehat{L}\widehat{\mathbf{z}}^{\ast} %
                     = - \widehat{\mathbf{1}}_{n}^{\top}\widehat{\mathbf{r}} = \rho$.
    \end{itemize}
    Thus, we can conclude, from the above deduction, that $\mathbf{z}^{\ast}$ is a solution of the linear system \eqref{eqn:sinsysm} %
    satisfying the orthogonal condition $\mathbf{1}_{n}^{\top}\mathbf{z}^{\ast} = 0$.
\end{proof}

\subsection{Deflation Technique}\label{sec:deflting}

Deflating converged eigenpairs is commonly used for solving the eigenvalue problems \cite{P98,S01,GvL12,HGLY14}. %
By deflating the approximate eigenpairs that we captured, %
we can always focus on finding the smallest positive eigenvalue and its corresponding eigenvector. %

Concretely, suppose that we have some converged positive eigenvalues associated with unit eigenvectors, %
say $(\lambda_1,\mathbf{v}_1),\ldots,(\lambda_c,\mathbf{v}_c)$. Note that $\lambda_0 = 0$. 
Consider the deflating matrix $\widetilde{L}$ defined by
\begin{equation}\label{eqn:deflating}
  \widetilde{L} := L + \delta VV^\top, \ \text{where}\ %
  \delta > 0 \text{ and } V = \begin{bmatrix} \mathbf{v}_1 & \cdots & \mathbf{v}_c \end{bmatrix}\in \mathbb{R}^{n \times c}.
\end{equation}
The constant $\delta > 0$ indicates how far we throw the captured eigenvalues.

\begin{theorem}\label{thm:deflating}
For the deflating graph Laplacian matrix $\widetilde{L}$ defined in \eqref{eqn:deflating}, we know that
    \begin{itemize}
        \item[(i)]  $\widetilde{L}$ is symmetric positive semi-definite with $\text{\normalfont{null}}(\widetilde{L}) = \text{\normalfont{span}}\{\mathbf{1}_{n}\}$.
        \item[(ii)] The deflated eigenvalue problem,
                     \begin{equation}\label{eqn:DGLEP}
                        \widetilde{L}\mathbf{x} = (L + \delta VV^\top)\mathbf{x} = \lambda \mathbf{x},
                     \end{equation}
                     preserves all eigenpairs of the original eigenvalue problem except for $\mu \in \{\lambda_1,\ldots,\lambda_c\}$. %
                     Instead, these eigenvalues are transformed into $\lambda_j + \delta $, which is a eigenvalue of \eqref{eqn:DGLEP} %
                     associated with the same eigenvector $\mathbf{v}_j$, $j = 1, \ldots, c$.
    \end{itemize}
\end{theorem}

\begin{proof}
    \begin{itemize}
        \item[(i)] It is obvious that $\widetilde{L}$ is symmetric and positive semi-definite from the facts that %
        $L$ as well as $\delta VV^\top$ are both symmetric and positive semi-definite matrices. %
        Note that columns of $V$ are eigenvectors of $L$ corresponding positive eigenvalues so that we have $V^{\top}\mathbf{1}_{n} = \mathbf{0}$.

        \item[(ii)] Observe that if $(\mu, \mathbf{u})$ is an eigenpair of $L$ with $\mu \notin \{\lambda_1,\ldots,\lambda_c\}$ then $V^\top\mathbf{u} = \mathbf{0}$ %
        since eigenvectors of a real symmetric matrix corresponding to different eigenvalues are orthogonal to each other. %
        Thus, we have $\widetilde{L}\mathbf{u} = L\mathbf{u} = \mu\mathbf{u}$ which implies that $(\mu, \mathbf{u})$ is also an eigenpair of $\widetilde{L}$. %
        In addition, for each captured eigenpair $(\lambda_j, \mathbf{v}_j)$, the pair $(\lambda_j + \delta, \mathbf{v}_j)$ will be an eigenpair of $\widetilde{L}$ since %
        \begin{equation*}
            \widetilde{L}\mathbf{v}_j = L\mathbf{v}_j + \delta VV^\top\mathbf{v}_j 
                                      = (\lambda_j + \delta)\mathbf{v}_j, %
            \quad j = 1, \ldots, c.
        \end{equation*}
    \end{itemize}
\end{proof}

\subsection{Shift-Invert of the Deflating GLEP \eqref{eqn:DGLEP}}\label{sec:sidevp}

In spite of the deflation technique allows us to throw away convergent eigenvalues and focus on the smallest positive one instead, %
the trick of shift-invert is still needed because the next interested eigenvalue is farther away from the zero. %

To the end, under the same assumptions in Section~\ref{sec:deflting}, we eventually need to deal with an eigenvalue problem of the form:
\begin{equation}\label{eqn:sdsep}
    (L + \delta VV^\top - \sigma I)^{-1} \mathbf{x} = \theta\mathbf{x},
\end{equation}
where $\delta > 0$ is a constant, $V\in \mathbb{R}^{n \times c}$ is given as in \eqref{eqn:deflating} while %
$\sigma$ is an appropriate shift value that close to the smallest positive eigenvalue of $L + \delta VV^\top$, i.e., $0 < \sigma \approx \lambda_{c+1}$. %
Note that $\mathbf{1}_{n}$ is also the kernel vector of $\widetilde{L}$ (cf. Theorem~\ref{thm:deflating}). %
In this case, the relating linear system is
\begin{equation}\label{eqn:sdlinsys}
\left\{
    \begin{array}{l}
        (L + \delta VV^\top - \sigma I)\mathbf{z} = \mathbf{r} \quad \text{with} \quad \mathbf{1}_{n}^{\top}\mathbf{z} = 0, \smallskip\\
        \text{where } L^{\top} = L, \ L\mathbf{1}_{n} = \mathbf{0},\ V^\top\mathbf{1}_{n} = \mathbf{0},\ \text{and}\ \mathbf{1}_{n}^{\top}\mathbf{r} = 0.%
    \end{array}%
\right.
\end{equation}

\begin{remark}\label{rmk:rortho1}
  With the requirement that $\mathbf{z}$ is perpendicular to $\mathbf{1}_{n}$, the same arguments as in Lemma~\ref{lem:rortho1} show that %
  $\mathbf{1}_{n}^{\top}\mathbf{r} = 0$ is also a necessary condition provided \eqref{eqn:sdlinsys} is a consistent system. %
  See also \eqref{eqn:sinsysm}.
\end{remark}

As mentioned above, to find the smallest deflated eigenvalue problem \eqref{eqn:DGLEP}, the linear system of the form \eqref{eqn:sdlinsys} needs to be solved.
Similar to the results proposed in Section~\ref{sec:trimming}, we have the following theorem %
when the trimming technique is applied to a shift-invert deflating eigenvalue problem \eqref{eqn:sdsep}. %

\begin{remark}
To maintain consistency and conciseness of notations, we use $\widehat{I}_n$ to denote the identity matrix of order $n-1$ and %
$\widehat{\mathbf{1}}_{n}$ is the all-one vector with dimension $n-1$.
\end{remark}

\begin{theorem}\label{thm:trimming_1}
    Suppose $\widehat{\mathbf{z}}^{\ast}\in \mathbb{R}^{n-1}$ is the solution of the linear system %
    \begin{equation}\label{eqn:dlinsys_i}
        \left( \left(\widehat{L} - \sigma \widehat{I}_n \right) + \tfrac{\sigma}{n} \widehat{\mathbf{1}}_{n}\widehat{\mathbf{1}}_{n}^{\top} %
         + \delta\widehat{V} \widehat{V}^{\top} \right) \widehat{\mathbf{z}} = \widehat{\mathbf{r}},
    \end{equation}
    where $\widehat{L}$ is a matrix of order $n-1$ as in \eqref{eqn:Lhati} with a specific index $i$, %
    \begin{equation*}
    \widehat{V} = \begin{bmatrix} V(1\text{\normalfont{:}}i-1,:) \\ V(i+1\text{\normalfont{:}}n,:) \end{bmatrix}\in \mathbb{R}^{(n-1)\times c}
    \quad \text{and} \quad
    \widehat{\mathbf{r}} = \begin{bmatrix} \mathbf{r}(1\text{\normalfont{:}}i-1) \\ \mathbf{r}(i+1\text{\normalfont{:}}n) \end{bmatrix}\in \mathbb{R}^{n-1}
    \end{equation*}
    are, respectively, the matrix and vector obtained from $V$ in \eqref{eqn:deflating} and $\mathbf{r}$ in \eqref{eqn:sdlinsys} by deleting the $i$th row individually.
    Then, the $n$-vector $\mathbf{z}^{\ast}$,%
    \begin{equation*}\mathbf{z}^{\ast} :=
        \begin{bmatrix}
            \widehat{\mathbf{z}}^{\ast}(1\text{\normalfont{:}}i-1) \\ %
            0 \\ %
            \widehat{\mathbf{z}}^{\ast} (i\text{\normalfont{:}}n-1)
        \end{bmatrix}
        - \frac{\widehat{\mathbf{1}}_{n}^{\top}\widehat{\mathbf{z}}^{\ast}}{n}\mathbf{1}_{n},
    \end{equation*}
    with the same form constructed in \eqref{eqn:z_i}, is a solution of \eqref{eqn:sdlinsys} with $\mathbf{1}_{n}^{\top}\mathbf{z}^{\ast} = 0$.
\end{theorem}

\begin{remark}\label{rem:linsys}
    If the matrix $V$ is empty and the shift $\sigma$ is zero simultaneously, %
    the linear system \eqref{eqn:sdlinsys} that we are dealing with goes back to the problem \eqref{eqn:sinsysm} so that %
    these two systems \eqref{eqn:dlinsys_i} and \eqref{eqn:tlinsys} are exact the same one. %
    Moreover, if anything about the smallest positive eigenvalue is known in advance, %
    which means if we can have a particular $\sigma > 0$ for finding the smallest positive value, %
    the linear system \eqref{eqn:tlinsys} should be, according to \eqref{eqn:dlinsys_i}, modified by
    \begin{equation}\label{eqn:trim_sigma}
    \left(\Big(\widehat{L} - \sigma \widehat{I}_n\Big) + \tfrac{\sigma}{n}\widehat{\mathbf{1}}_{n}\widehat{\mathbf{1}}_{n}^{\top} \right) %
    \widehat{\mathbf{z}} = \widehat{\mathbf{r}}. %
    \end{equation}
\end{remark}

\begin{proof}[Proof of Theorem~\ref{thm:trimming_1}]
    By convention, we also delete the last row and column of $\widetilde{L}$. %
    That is, $\widehat{L}$, $\widehat{\mathbf{z}}$ $\zeta$, $\widehat{\mathbf{r}}$ and $\rho$ are defined as in \eqref{eqn:Lhzhrh} with dimension $n-1$, %
    and denote $\widehat{V}\in \mathbb{R}^{(n-1)\times c}$ the matrix obtained from $V$ in \eqref{eqn:deflating} by deleting the last row. Set %
    \begin{equation*}\label{trimmedmat}
        J = I + \mathbf{e}_{n}\mathbf{1}_{n}^{\top} - \mathbf{e}_{n}\mathbf{e}_{n}^{\top}
        = \begin{bmatrix}
                1      & 0       & \cdots  & 0       & 0 \\
                0      & 1       & \cdots  & 0       & 0 \\
                0      & 0       & \ddots  & \vdots  & \vdots  \\
                \vdots & \vdots  &         & 1       & 0 \\
                1      & 1       & \cdots  & 1       & 1%
            \end{bmatrix}\in\mathbb{R}^{n \times n}
    \end{equation*} %
    so that we have $J^{-\top} = I_{n}  - \mathbf{1}_{n}\mathbf{e}_{n}^{\top} + \mathbf{e}_{n}\mathbf{e}_{n}^{\top}$. %
    Then, we can verify, under the assumptions in \eqref{eqn:sinsysm} and the particular requirement %
    $0 = \mathbf{1}_{n}^{\top}\mathbf{z} = \widehat{\mathbf{1}}_{n}^{\top}\widehat{\mathbf{z}} + \zeta$, that %
\begin{align}
    \left( L + \delta VV^{\top} - \sigma I_{n} \right) \mathbf{z} = \mathbf{r} %
    \Leftrightarrow& \left. \left( J \left( L + \delta VV^{\top} - \sigma I_{n} \right) J^{\top} \right) \left( J^{-\top}\mathbf{z} \right) = J \mathbf{r}\right. \notag\\
    \Leftrightarrow& \begin{pmat}[{|.}]
                        \widehat{L} + \delta\widehat{V}\widehat{V}^{\top} - \sigma \widehat{I}_n & -\sigma\widehat{\mathbf{1}}_{n} \cr \-
                        -\sigma \widehat{\mathbf{1}}_{n}^{\top} & -\sigma n \cr
                     \end{pmat} \begin{pmat}[{.}] \widehat{\mathbf{z}}-\zeta \widehat{\mathbf{1}}_{n} \cr \- \zeta \cr \end{pmat}
                    = \begin{pmat}[{.}] \widehat{\mathbf{r}} \cr \- 0 \cr \end{pmat}   \notag\\
    \Leftrightarrow& \begin{pmat}[{.}]
                        \left( \widehat{L} + \delta\widehat{V}\widehat{V}^{\top} - \sigma \widehat{I}_n \right)
                        \left( \widehat{\mathbf{z}}-\zeta \widehat{\mathbf{1}}_{n} \right) - \sigma \zeta \widehat{\mathbf{1}}_{n} \cr \-
                        -\sigma \left( \widehat{\mathbf{1}}_{n}^{\top}\widehat{\mathbf{z}} + \zeta \right)\cr
                       \end{pmat}
                    = \begin{pmat}[{.}] \widehat{\mathbf{r}} \cr \- 0 \cr \end{pmat} \notag \\
    \Leftrightarrow& \Big(\widehat{L} + \delta\widehat{V}\widehat{V}^{\top} - \sigma \widehat{I}_n \Big) %
                    \Big(\widehat{I}_n + \widehat{\mathbf{1}}_{n}\widehat{\mathbf{1}}_{n}^{\top}\Big)\widehat{\mathbf{z}} +%
                    \sigma \widehat{\mathbf{1}}_{n}\widehat{\mathbf{1}}_{n}^{\top}\widehat{\mathbf{z}} = \widehat{\mathbf{r}} \notag \\
    \Leftrightarrow& \left(\Big(\widehat{L} - \sigma \widehat{I}\Big) + \tfrac{\sigma}{n}\widehat{\mathbf{1}}_{n} \widehat{\mathbf{1}}_{n}^{\top} %
                           + \delta\widehat{V}\widehat{V}^{\top} \right) %
                           \big(\widehat{I}_n + \widehat{\mathbf{1}}_{n}\widehat{\mathbf{1}}_{n}^{\top}\big)\widehat{\mathbf{z}} %
                           = \widehat{\mathbf{r}}. \notag
\end{align}
The last equation holds thanks to the equality \eqref{eqn:rank1inv} which allows to extract the factor %
$\big(\widehat{I}_n + \widehat{\mathbf{1}}_{n}\widehat{\mathbf{1}}_{n}^{\top}\big)\widehat{\mathbf{z}}$. %
As in the proof of Theorem~\ref{thm:trimming}, we see that if $\widehat{\mathbf{z}}^{\ast}$ is the solution of
\begin{equation*}
    \left(\Big(\widehat{L} - \sigma \widehat{I}_n\Big) + \tfrac{\sigma}{n}\widehat{\mathbf{1}}_{n} \widehat{\mathbf{1}}_{n}^{\top} +%
          \delta\widehat{V}\widehat{V}^{\top}  \right) \widehat{\mathbf{z}} = \widehat{\mathbf{r}},
\end{equation*}
then the $n$-vector \eqref{eqn:z} will be a solution of \eqref{eqn:sdlinsys} satisfying $\mathbf{1}_{n}^{\top} \mathbf{z}^{\ast} = 0$. %
\end{proof}

\section{Solving the Graph Laplacian Eigenvalue Problem}\label{sec:Implement}

We first consider a general case for dealing with a large and sparse eigenvalue problem. %
Suppose we are interested in some eigenvalues that closed to a target $\sigma$ of an large and sparse eigenvalue problem %
\begin{equation}\label{eqn:Ax_lx}
    A\mathbf{x} = \lambda \mathbf{x},\quad A \in \mathbb{R}^{n \times n}.
\end{equation}
A traditional method to solve such a problem is the Shift-Invert Arnoldi (SIA) method. %
SIA is a projection method that applies the Arnoldi method to the operator $(A - \sigma I)^{-1}$ %
for computing some eigenvalues nearest to $\sigma$ and the associated eigenvectors. %
In general, the project subspace is the so-called order-$k$ $(k \ll n)$ Krylov subspace generated by $B_{\sigma} : =(A-\sigma I)^{-1}$ %
and a unit vector $\mathbf{u}_1$:
\begin{equation}\label{eqn:krylov}
    \mathcal{K}_{k}(B_{\sigma}, \mathbf{u}_1) := %
    \{ \mathbf{u}_1, B_{\sigma}\mathbf{u}_1, B_{\sigma}^{2}\mathbf{u}_1, \ldots, B_{\sigma}^{k-1}\mathbf{u}_1 \}. %
\end{equation}
Once we have an $n \times k$ matrix $U_{k}:=[\mathbf{u}_{1},\mathbf{u}_{2},\ldots,\mathbf{u}_{k}]$ whose columns form an orthonormal basis of \eqref{eqn:krylov}. %
SIA will solve the eigenpair of $H_{k} = U_{k}^{\top}AU_{k}$, say $(\theta,\mathbf{s})$, that we are interested, %
and then recover the Ritz pair $(\theta,U_k\mathbf{s})$ of $B_{\sigma}$ to $(\sigma + 1/\theta, U_k\mathbf{s})$ as an approximate eigenpair of $A$. %
After that, if the accuracy is not good enough, the Krylov subspace \eqref{eqn:krylov} will be expanded. %
To this end, one has to find the solution $\mathbf{z}$ of the linear system
\begin{equation}\label{eqn:siarlinsys}
    (A - \sigma I)\mathbf{z} = \mathbf{u}_{k}
\end{equation}
and then orthogonalize the vector $\mathbf{z}$ against $U_{k}$ to generate the next basis vector $\mathbf{u}_{k+1}$ of $\mathcal{K}_{k+1}\big(B_{\sigma},\mathbf{u}_1\big)$. %

Direct methods, such as the LU or the Cholesky factorizations, are used to solve linear system \eqref{eqn:siarlinsys} exactly %
for the construction of the Arnoldi or Lanczos decomposition. %
{\it However, for a large matrix $A$, such a factorization is not feasible in general, and only iterative solvers are viable}. %
This difficulty motivates us to introduce the Shift-Invert Residual Arnoldi (SIRA) \cite{L07,LS07} method for the use of %
{\it inexactly} solving the inner linear systems.

\subsection{Review of the SIRA Method}

SIRA (Algorithm~\ref{alg:SIRA}) is an alternative applied to the matrix $B_{\sigma} := (A - \sigma I)^{-1}$ %
for computing a few eigenvalues of \eqref{eqn:Ax_lx} that is closed to $\sigma$. %

\begin{algorithm}[h]\caption{SIRA method with the target $\sigma$}\label{alg:SIRA}
\begin{algorithmic}[1]
    \REQUIRE Given a square matrix $A$ with size $n$; the target $\sigma$ of interested location; %
             an $n\times k$ matrix $U$ $(1\leq k \ll n)$ whose columns form an orthonormal basis of the order-$k$ Krylov subspace \eqref{eqn:krylov}; %
             and the tolerance $\varepsilon$.
    \ENSURE  The eigenpair $(\lambda, \mathbf{v})$ of $A$. %
    \REPEAT
        \STATE {\color{OliveGreen}{{\bf \% Subspace Extraction}}}
        \STATE Compute the Rayleight quotient $H = U^\top A U$.
        \STATE Let $(\theta,\mathbf{s})$ be an eigenpair of $H$, where $\theta \approx \lambda$.
        \STATE Compute the residual $\mathbf{r} = A\mathbf{y} - \theta \mathbf{y}$, where $\mathbf{y} = U\mathbf{s}$.
            \IF{$|\mathbf{r}| < \varepsilon$}
                \STATE {\bf return} $(\lambda,\mathbf{v})$, where $\lambda = \theta$ and $\mathbf{v} = \mathbf{y}$.
            \ELSE
                \STATE {\color{OliveGreen}{{\bf \% Subspace Expansion}}}
                \STATE Solve the linear system $(A-\sigma I)\mathbf{z} = \mathbf{r}$.
                \STATE Orthonormalize $\mathbf{z}$ against $U$ to obtain $\mathbf{u}$.
                \STATE Update $U = [U\ \mathbf{u}]$.
            \ENDIF
        \UNTIL{Capture the Ritz pair (approximate eigenpair) of $A$ with $\theta \approx \sigma$.}
\end{algorithmic}
\end{algorithm}

In the step of subspace extraction, SIRA takes the Ritz pair $(\theta,U_{k}\mathbf{s})$ as an approximate eigenpair of $A$ directly;
in the step of subspace expansion, SIRA solves the linear system
\begin{equation}\label{eqn:linsysRIRA}
    (A-\sigma I)\mathbf{z} = \mathbf{r}, \quad \mathbf{r} := (A - \theta I) U_{k}\mathbf{s}
\end{equation}
for the purpose of getting the next basis vector of $\mathcal{K}_{k+1}(B_{\sigma},\mathbf{u}_1)$. %
In summary, even through the projection subspace $\mathcal{K}_k(B_{\sigma},\mathbf{u}_1)$ of SIA and SIRA is the same for the identical unit vector $\mathbf{u}_1$, these two methods generally obtain different approximations \cite{JL14}. %

To use the SIRA method for finding a few eigenvalues nearest to $0$ and the associated eigenvectors of the GLEP~\eqref{eqn:GLEP}, %
we successively dig the desired pairs from a Krylov subspace and then expand this searching subspace if the results are not yet satisfactory. %
In conclusion, we have to solve a bunch of linear systems as in line~10 of the SIRA algorithm. %

Note that the right-hand side of the linear system in SIRA automatically satisfies the necessary condition of Lemma~\ref{lem:rortho1}. %
Owing to $\mathbf{r}$ in \eqref{eqn:linsysRIRA} is the residual vector and the $(\theta,U_{k}\mathbf{s})$ is a Ritz pair of $L$
with $\theta > 0$ and $\mathbf{1}_{n}^{\top}U_{k} = \mathbf{0}$, it implies that
\begin{equation*}
\mathbf{1}_{n}^{\top}\mathbf{r} = \mathbf{1}_{n}^{\top}(LU_{k}\mathbf{s} - \theta U_{k}\mathbf{s}) = 0. %
\end{equation*}
However, for large-scale applications, using direct methods to solve \eqref{eqn:linsysRIRA} in SIRA is still expensive in memory and time consuming.  So, in general, only iterative solvers are viable. %

It is worth mentioning that Lee \cite{L07} as well as Lee and Stewart \cite{LS07} made some analysis 
and indicated that the SIRA method may still work well in spite of the low or modest accuracy at each step for solutions of the linear systems \eqref{eqn:linsysRIRA}. %
This leads to the {\it inexact} SIRA method. Recently, Jia and Li \cite{JL14} proved that the inexact SIRA mimics the exact SIRA well %
when the relative error of the approximate solution of \eqref{eqn:linsysRIRA} is modestly small at each iteration.

Therefore, to solve the GLEP \eqref{eqn:GLEP}, we integrate an inner-outer iterative method, built-up by SIRA \cite{L07} as the iterative outer eigensolver %
and {\it inexact} solving inner linear systems \eqref{eqn:sinsysm} as well as \eqref{eqn:sdlinsys}.

\subsection{The Linear Systems in SIRA}\label{sec:cmg}

To our problem, we need to solve the singular linear system \eqref{eqn:sinsysm} as well as \eqref{eqn:sdlinsys}. %
In Section~\ref{sec:trimming}, we propose a trimming technique to remedy the singularity of $L$, %
and to solve an $(n-1) \times (n-1)$ linear system \eqref{eqn:tlinsys} instead. %
Then the resulting vector can be converted into a solution of \eqref{eqn:sinsysm} that is orthogonal to the kernel vector $\mathbf{1}_{n}$. %
If we attempt to find more than one eigenpair, the deflation method introduced in Section~\ref{sec:deflting} can be used to exclude the influence of convergent eigenpairs. %
In this case, we in fact face a deflated eigenvalue problem \eqref{eqn:DGLEP} and need to solve the linear system \eqref{eqn:dlinsys_i}. %

\begin{remark}\label{rmk:low_rank}
    When $n$ is very large, we can omit the low-rank term, $\tfrac{\sigma}{n}\widehat{\mathbf{1}}_{n}\widehat{\mathbf{1}}_{n}^\top$, %
    appearing in \eqref{eqn:dlinsys_i} and \eqref{eqn:trim_sigma}. At this moment, we turn to solve %
    \begin{equation*}
        \widehat{L}_{\sigma} \widehat{\mathbf{z}} = \widehat{\mathbf{r}} \quad \text{and}\quad %
        \left(\widehat{L}_{\sigma} + \delta\widehat{V}\widehat{V}^{\top}\right) \widehat{\mathbf{z}} = \widehat{\mathbf{r}}
    \end{equation*}
    instead of \eqref{eqn:trim_sigma} and \eqref{eqn:dlinsys_i}, respectively, where $\widehat{L}_{\sigma} := \widehat{L}  - \sigma\widehat{I}_{n}$. %
\end{remark}

Choosing suitable preconditioners plays an important role to get a decent performance for solving a linear system with iterative methods. %
Based on the the idea of Vaidya \cite{V91} using graph theory for iterative methods, %
Speilman and Teng \cite{ST04} gave a combinatorial preconditioner by graph sparsification and proposed the first near-linear time Symmetric Diagonal Dominate\footnote{A square matrix $A$ is said to be symmetric diagonal dominate if $A^\top = A$ and $A_{jj} \geq \sum_{j\neq k}\left\vert A_{jk}\right\vert$.} (SDD) solver. %
Koutis {\it et al.} \cite{KMT11,KMP11} proposed the construction of Combinatorial MultiGrid, referred to as CMG, preconditioning chain and %
applied the CMG preconditioner to deal with optimization problems in computer vision. %
The algorithm CMG reduces solving general SDD systems to solve systems in graph Laplacian matrices. %
Given a graph, they construct by adding carefully chosen sets of edges to obtain a low-stretch spanning tree and to get a sequence of %
logarithmically many successively sparser graphs that approximate it as a preconditioner of iterative methods, %
such as the preconditioned conjugate gradient method or the minimal residual method. %
For a more in-depth discussion of this work, we refer to \cite{KOSZ13,ST14}.


The matrix $\widehat{L}$ in \eqref{eqn:Lhati} is also indeed a SDD matrix so that we take the corresponding CMG matrix %
$\widehat{M}$ as a preconditioner in \eqref{eqn:tlinsys}. %
However, $\widehat{M}$ is not a suitable preconditioner if $\sigma$ is nonzero and/or $V$ is nonempty because our goal is to solve the linear systems \eqref{eqn:dlinsys_i} and \eqref{eqn:trim_sigma}. %
In such cases, as the suggestion in \cite{HHLW14}, their preconditioners should involve the deflation terms. %
Moreover, according to Remark~\ref{rmk:low_rank}, we discard the rank-one correction $\tfrac{\sigma}{n}\widehat{\mathbf{1}}_{n}\widehat{\mathbf{1}}_{n}^\top$ %
in both equations if $n$ is very large. As a consequence, for large $n$, we take $\widehat{M}_{\sigma}$ and $\widehat{M}_{\sigma} + \delta\widehat{V}\widehat{V}^{\top}$ %
as preconditioners for \eqref{eqn:trim_sigma} and \eqref{eqn:dlinsys_i}, respectively, where $\widehat{M}_{\sigma} \approx \widehat{L} - \sigma \widehat{I}_n$. %
Note that $\sigma$ should be chosen carefully to preserve $\widehat{L} - \sigma \widehat{I}_n$ to be a SDD matrix. %

Suppose that we obtain the first $c$ (smallest) eigenpairs of $L$. Let $\Lambda\in\mathbb{R}^{c\times c}$ be the eigenvalue matrix and %
$V\in \mathbb{R}^{n\times c}$ be the corresponding eigenvector matrix whose columns are orthonormal to each other. %
By means of the Sherman-Morrison-Woodbury formula \cite{GvL12}, %
the inverse of $\widehat{M}_{\sigma} + \delta\widehat{V}\widehat{V}^{\top}$ with $\widehat{V}\in\mathbb{R}^{(n-1)\times c}$ is given by %
\begin{equation}\label{eqn:SMWF}
    \left(\widehat{M}_{\sigma} + \delta\widehat{V}\widehat{V}^{\top}\right)^{-1}
  = \left(\widehat{I}_n - \delta\widehat{M}_{\sigma}^{-1}\widehat{V}%
   (I_{c} + \delta\widehat{V}^{\top}\widehat{M}_{\sigma}^{-1}\widehat{V})^{-1}\widehat{V}^{\top}\right)\widehat{M}_{\sigma}^{-1},
\end{equation}
where $\widehat{I}_n = I_{n-1}$. %


\begin{remark}[{\bf Strategy for Preconditioners}]\label{cor:large_n_systems}
  When $n$ is large, with the previous notations, the choice of preconditioners to solve the linear systems %
  in {\normalfont SIRA} by iterative methods can be summarized as follows.
  \begin{itemize}
    \item[(i)] To the case \eqref{eqn:tlinsys}, we choose $\widehat{M}  \approx \widehat{L}$ as a preconditioner.
    \item[(ii)] To the case \eqref{eqn:trim_sigma}, we turn to consider the linear system %
                $\left(\widehat{L}  - \sigma\widehat{I}_{n}\right)\widehat{\mathbf{z}} = \widehat{\mathbf{r}} $ and choose $\widehat{M}_{\sigma} \approx \widehat{L}  - \sigma \widehat{I}_{n}$ as a preconditioner.
    \item[(iii)] To the case \eqref{eqn:dlinsys_i}, we consider the linear system %
                    $\left(\left(\widehat{L}  - \sigma \widehat{I}_n\right) + \delta \widehat{V} \widehat{V} ^{\top}\right) %
                    \widehat{\mathbf{z}} = \widehat{\mathbf{r}} $ %
                and choose $\widehat{M}_{\sigma} + \delta \widehat{V} \widehat{V} ^{\top}$ as a preconditioner %
                with the aid of the formula \eqref{eqn:SMWF}.
  \end{itemize}
\end{remark}

\subsection{Integrated SIRA with Trimming and Deflating Techniques for GLEPs}

Algorithm~\ref{alg:SIRAtd}, called iSIRA, summarizes the SIRA method combined with techniques of trimming (Section~\ref{sec:trimming}) and deflation (Section~\ref{sec:deflting}) for finding some smallest positive eigenvalues and associated eigenvectors of the GLEP \eqref{eqn:GLEP}. %
The techniques of trimming, deflation and shift-invert enhance can be described as in the Figure~\ref{fig:tds}.

\begin{figure}
  \centering
  \includegraphics[width=1\textwidth]{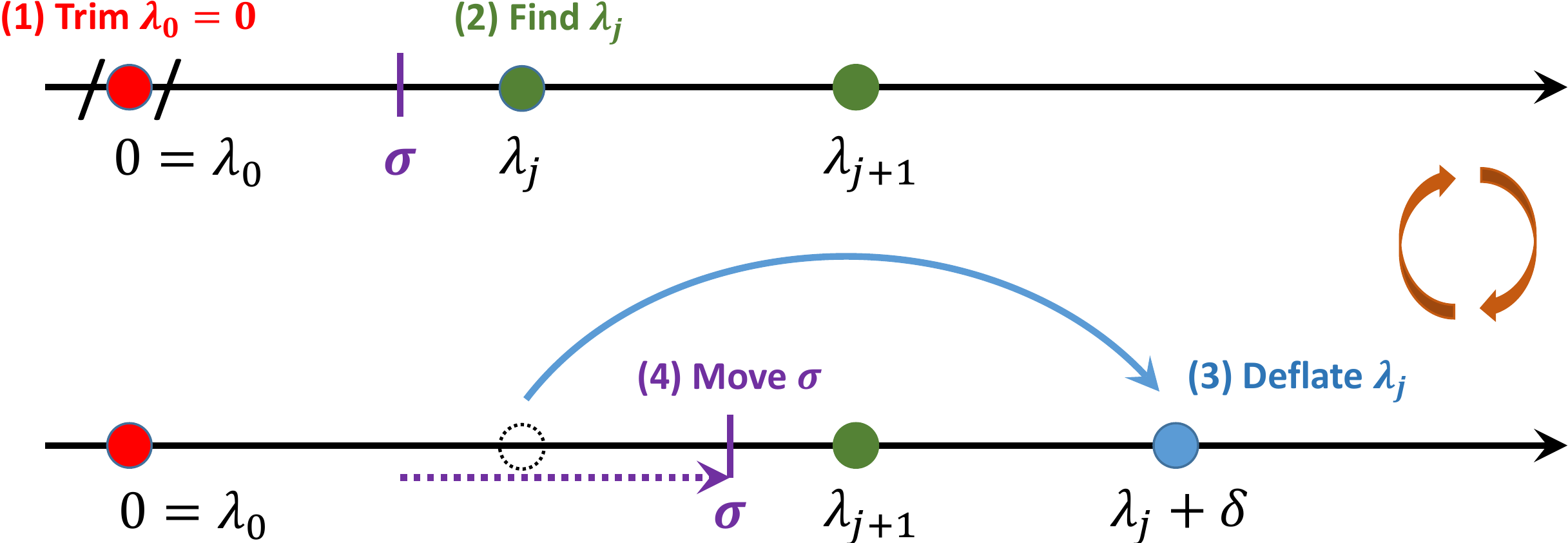}
  \caption{(1) Trim $\lambda_0 = 0$; (2) Find $\lambda_j\ (j>0)$; (3) Deflate $\lambda_j$ to $\lambda_j + \delta$; (4) Move $\sigma$.}\label{fig:tds}
\end{figure}
\begin{algorithm}[!htb]\caption{iSIRA: integrated SIRA method for GLEP \eqref{eqn:GLEP}}\label{alg:SIRAtd}
\begin{algorithmic}[1]
    \REQUIRE An $n \times n$ graph Laplacian matrix $L$; a number of desired value $d$; an $n \times k_{0}$ column orthonormal matrix $U_{k_{0}}$, %
             $1 \leq k_{0} < m$, that each column of $U_{k_{0}}$ is also orthogonal to $\mathbf{1}_{n}$; trimmed index $1 \leq i \leq n$; %
             a deflating scalar $\delta > 0$; target shift $\sigma(=0)$; tolerances $\varepsilon$; %
             maximum dimension of the search space $m$ and the restart size of the subspace $q$. %
             Note that $1\leq d \leq q < m \ll n$.
    \ENSURE The first $d$ smallest positive pairs $(\lambda_j,\mathbf{v}_j)$ of $L$, $j = 1,\ldots, d$. %
    \STATE Set $D = [\ ]$, $V = [\ ]$ (empty arrays), $k = k_{0}$ and $\text{flag} = 0$.
    \FOR{$j = 1, \ldots , d$}
        \WHILE{$\text{flag}=0$}
            \STATE Compute $W_{k} = (L + \delta VV^{\top})U_{k}$ and $H_{k} = U_{k}^{\top}W_{k}$.
            \STATE {\color{OliveGreen}{{\bf \% Subspace Extraction}}}
            \STATE Compute all eigenpairs of $H_{k}$, say $\{(\theta_{1},\mathbf{s}_{1}),\ldots,(\theta_{k},\mathbf{s}_{k})\}$, where %
                   $\theta_{t}$ is in ascending order and $\|\mathbf{s}_{t}\|_{2} = 1$, $t = 1, \ldots, k$. Denote $S_{k} = [\mathbf{s}_{1}\ \cdots\ \mathbf{s}_{k}]$.
            \STATE Compute the residual $\mathbf{r} = W_k\mathbf{s}_{1} - \theta_{1}U_k\mathbf{s}_{1}$.
            \IF{$\|\mathbf{r}\|_2 < \varepsilon$}
                \STATE Update $D = [D\ \ \lambda_{j}]$ and $V = [V\ \ \mathbf{v}_{j}]$, where $\lambda_{j} = \theta_{1}$ and $\mathbf{v}_{j} = U_k\mathbf{s}_{1}$. %
                \STATE Set $\text{flag} = 1$.
            \ELSE
                \IF{$k = m$}
                    \STATE {\color{OliveGreen}{{\bf \% Subspace Restart}}}\newline
                        $U_{k} = U_{m}S_{k}(:,1\text{\normalfont{:}}q)$, $W_{k} = W_{m}S_{k}(:,1\text{\normalfont{:}}q)$, %
                        $H_{k} = \text{diag}\{\theta_{1},\ldots,\theta_{q}\}$.
                    \STATE Set $k = q$.
                \ENDIF
                \STATE {\color{OliveGreen}{{\bf \% Subspace Expansion}}}
                \STATE Find a solution $\mathbf{z}^{\ast}$ of \eqref{eqn:sinsysm} if $j = 1$; of \eqref{eqn:sdlinsys} if $j > 1$. \newline
                {\color{OliveGreen}{{\bf \% Singularity Trimming and Eigenvalue Deflating}}}\newline
                    (i)\ \ Solve the linear system \eqref{eqn:tlinsys} if $j = 1$; the system \eqref{eqn:dlinsys_i} if $j > 1$.\newline
                    (ii)\ Compute the $n$-vector $\mathbf{z}^{\ast}$ \eqref{eqn:z_i}.
                \STATE Orthonormalize $\mathbf{z}^{\ast}$ against $U_{k}$ (and $V$ if $j > 1$) to obtain $\mathbf{u}$.
                \STATE Expand $U_{k}$ by $U_{k+1} = \big[U_{k}\ \ \mathbf{u}\big]$. Set $k = k + 1$.
            \ENDIF
        \ENDWHILE
        \STATE {\color{OliveGreen}{{\bf \% Eigenvector Purging}}} \newline
            $U_{k} = U_{k}S_{k}(:,2\text{\normalfont{:}}k)$, $W_{k} = W_{k}S_{k}(:,2\text{\normalfont{:}}k)$, $H_{k} = \text{diag}\{\theta_{2},\ldots,\theta_{k}\}$.
        \STATE Move $\sigma \approx \theta_{2}$. Set $k = k - 1$ and $\text{flag} = 0$. %
    \ENDFOR
\end{algorithmic}
\end{algorithm}

For a given $k \geq 1$, as in line~6 of the iSIRA algorithm, let $U_{k}$ be the current searching subspace and %
$\{(\theta_1,\mathbf{s}_1), \cdots, (\theta_k,\mathbf{s}_k)\}$ be the set of eigenpairs of $H_k$ %
with the ascending order of eigenvalues and the unit norm of each eigenvector. %
The $k \times k$ column orthonormal matrix $S_{k} = \begin{bmatrix}\mathbf{s}_{1} & \cdots & \mathbf{s}_{k}\end{bmatrix}$ %
is the collection of the eigenvectors of $H_{k}$. %

We first explain the step -- {\it subspace restart} -- in line~12 to 15 of the iSIRA algorithm. %
Due to the storage requirements and computational costs, the order of the searching subspace $k$ can not be too large and shall be limited. %
That is to say, we have to shrink down the subspace in case the dimension is equal to the limit size $m$. This process is called {\it restart}. %
So, once we generate an $m$-dimension search space $U_{m}$, i.e. $k = m$, %
we will reduce $U_{m}$ to a $q$-dimensional matrix, $1 \leq q < m$, that preserves some useful information. %
To maintain the first $q$ eigeninformation, the matrices $U_m$, $W_m$ and $H_m$ in the iSIRA will be updated by %
\begin{equation*}
  U_{k} \leftarrow U_{m}S_{m}(:,1\text{\normalfont{:}}q),\ W_{k} \leftarrow W_{m}S_{m}(:,1\text{\normalfont{:}}q),\ H_{k} \leftarrow \text{\normalfont{diag}}\{\theta_{1},\ldots,\theta_{q}\},\
  k \leftarrow q.
\end{equation*}

We then discuss the step -- {\it eigenvector purging} -- in line~22 of the iSIRA algorithm. %
Suppose we have a convergent eigenpair $(\lambda_{j}, \mathbf{v}_{j})$ for some $1 \leq j \leq d$, %
where $\lambda_{j} = \theta_{1}$ and $\mathbf{v}_{j} = U_{k}\mathbf{s}_{1}$. %
The deflating action will throw $\lambda_{j}$ forward to $\lambda_{j} + \delta$, %
and the operation matrix will, for the subsequent loops, contain the deflating term $L + \delta \mathbf{v}_{j}\mathbf{v}_{j}^{\top}$. %
See Section~\ref{sec:deflting} and line~17 of the iSIRA algorithm. %
Furthermore, we can also in fact purge the direction of $\mathbf{v}_{j} = U_{k}\mathbf{s}_{1}$. %
Consequently, owing to $U_{k}^{\top}U_{k} = I_{k} = S_{k}^{\top}S_{k}$, we will update $U_{k}$, $W_{k}$ and $H_{k}$ as follows %
\begin{equation*}
  U_{k} \leftarrow U_{k}S_{k}(:,2\text{\normalfont{:}}k),\ W_{k} \leftarrow W_{k}S_{k}(:,2\text{\normalfont{:}}k),\ H_{k} \leftarrow \text{\normalfont{diag}}\{\theta_{2},\ldots,\theta_{k}\},\
  k \leftarrow k-1,
\end{equation*}
which are the modifications in line~22 of the iSIRA algorithm. Note that, in this case, the dimension of the searching subspace is {\it shrunk} to $k-1$. %

\section{Numerical Experiments}\label{sec:Numerical}

In this section, we demonstrate the efficiency and accuracy of the iSIRA algorithm (Algorithm~\ref{alg:SIRAtd}), %
to solve the GLEP \eqref{eqn:GLEP} for computing some smallest positive eigenvalues and associated eigenvectors. %
We use {\tt isira}, Integrated SIRA, to denote Algorithm~\ref{alg:SIRAtd}.

All computations in this section are carried out in MATLAB~2017a. %
For the hardware configuration, we use a DELL XPS~15~9560 laptop with an Intel~i7-7700HQ Quad Core, 16GB~RAM, and the Windows~10 operating system.
The data sets to our numerical experiments are provided within KONECT \cite{K13} (Koblenz Network Collection, \url{http://konect.uni-koblenz.de}) %
and we consider the undirected, unweighted networks. %

Moreover, to demonstrate the performance of our algorithm, we compare {\tt isira} with two methods. %
We call the MATLAB built-in function {\tt eigs} to compute the eigenvalues with smallest magnitudes %
using the diagonal perturbation $\tau I_{n}$, where $\tau > 0$ is a small perturbation such as $10^{-8}$ and $n$ is the amount of vertex. %
Note that we have to compute $11$ eigenvalues of $L_{\tau} := L + \tau I_{n}$ since the smallest one is equal to $\tau$. %

Beside, we also apply the idea of null-space deflation in \cite{HGLY14}, proposed by the part authors (Huang, Lin and Yau) of this paper, %
to the graph Laplacian matrix. In short, using a rank-two correction, the matrix $L$ is transformed to %
$L_{d} + \big[\mathbf{e}_{1} \ \mathbf{1}_{n}\big]\Big[{-\mathbf{e}_{1}^{\top} \atop \mathbf{1}_{n}^{\top}}\Big]$, %
where $L_{d} := L+\mathbf{e}_{1}\mathbf{e}_{1}^{\top}$.
Note that this method is based on the {\tt eigs} algorithm with the deflation of the zero eigenvalue and we call it as {\tt ndeigs} (Null-space Deflating {\tt eigs}). %
That is, we calculate some smallest positive eigenvalues using {\tt eigs} with a function handle specified %
how to solve the linear systems whose the coefficient matrix $L_{d}$ together with a rank-two correction.

As a consequence, these two methods need to compute the matrix factorizations of $L_{\tau}$ and $L_{d}$ in advance %
for solving the linear systems involved in {\tt eigs}. %
Moreover, the convergence tolerance of each eigensolver is set to be $10^{-8}$. %

\begin{remark}\label{rmk:trimID}
    We give some remarks on the trimming index and the deflating coefficient %
    indicated in Section~\ref{sec:trimming} and Section~\ref{sec:deflting}, respectively.
    \begin{itemize}
      \item[(i)] The choice of trimming index may be a hyperparameter. In practice, we select the one whose diagonal entry has the maximum degree so that there are more rows (columns) will has strictly diagonally dominant property\footnote{This means that the magnitude of diagonal entry is strictly greater than the sum of the absolute value of the rest elements in the same row.}. In Example~\ref{exp:trimID}, we will present some comparison results. %
      \item[(ii)] We choose $\delta = 2 \cdot \max\{ \text{\normalfont diag}(L) \}$ to be the deflating coefficient in Theorem~\ref{thm:deflating} based on the {\normalfont Gershgorin} Theorem \cite{GvL12}. %
          This can ensure that deflating eigenvalues will not fall inside the spectrum of $L$ and still keep the order of magnitude. %
    \end{itemize}
\end{remark}

\begin{example}\label{ex:konect}
    We compute the first ten smallest eigenvalues of the networks from {\normalfont KONECT} with vertex size larger than 100,000. %
    The first five data in Table~\ref{tab:konect_sira} are connected networks and the rest of networks have more than one connected component. %
    For the latter, we first identify which component of each vertex in the graph belongs to and then take out %
    the subnetwork that has the most members. %
\end{example}
\begin{table}[!ht]
    \centering
    \begin{tabular}{c|c|c|D{.}{.}{2}D{.}{.}{2}D{.}{.}{2}}
    \toprule
    \multirow{2}{*}{Code} & \multirow{2}{*}{$n$} & \multirow{2}{*}{$nnz$} & \multicolumn{3}{c}{Time (sec)} \\ \cline{4-6}
                          & & & \multicolumn{1}{c}{\tt isira} & \multicolumn{1}{c}{\tt ndeigs} & \multicolumn{1}{c}{\tt eigs}\\
    \hline
    \hline
    \rowcolor[gray]{.9}
    {\bf LM}  & \multicolumn{1}{r|}{  104,103} & \multicolumn{1}{r|}{4,490,269}  &  72.3 & 153.4 & 1856.0                                      \\
    \rowcolor[gray]{1}
    {\bf GW}  & \multicolumn{1}{r|}{  196,591} & \multicolumn{1}{r|}{2,097,245}  &  51.6 &  61.3 & 224.7                                       \\
    \rowcolor[gray]{.9}
    {\bf CA}  & \multicolumn{1}{r|}{  334,863} & \multicolumn{1}{r|}{ 2,186,607} &  66.8 &  67.6 & 203.3                                       \\
    \rowcolor[gray]{1}
    {\bf CY}  & \multicolumn{1}{r|}{1,134,890} & \multicolumn{1}{r|}{ 7,110,138} & 315.3 & \multicolumn{1}{c}{---} & \multicolumn{1}{c}{---}\\
    \rowcolor[gray]{.9}
    {\bf HY}  & \multicolumn{1}{r|}{1,402,673} & \multicolumn{1}{r|}{ 6,957,511} & 223.3 & \multicolumn{1}{c}{---} & \multicolumn{1}{c}{---}\\
    \rowcolor[gray]{1}
    {\bf FX}  & \multicolumn{1}{r|}{2,523,386} & \multicolumn{1}{r|}{18,360,988} & 740.4 & \multicolumn{1}{c}{---}  & \multicolumn{1}{c}{---}\\
    \hline
    \hline
    \rowcolor[gray]{.9}
    {\bf FI}  &\multicolumn{1}{r|}{  105,722} & \multicolumn{1}{r|}{ 4,739,058}  &   54.8 & 61.6                    & 404.7                   \\
    \rowcolor[gray]{1}
    {\bf SK}  &\multicolumn{1}{r|}{1,694,616} & \multicolumn{1}{r|}{23,883,034}  &  660.6 & 10310.5                 & \multicolumn{1}{c}{---} \\
    \rowcolor[gray]{.9}
    {\bf YT}  &\multicolumn{1}{r|}{3,216,075} & \multicolumn{1}{r|}{21,955,823}  & 1068.0 & \multicolumn{1}{c}{---} & \multicolumn{1}{c}{---} \\
    \bottomrule
    \end{tabular}
    \caption{The CPU time of the eigensolvers for networks. The symbol --- means that 16 GB memory was not enough for computing the required matrix factorizations. %
    The first column is the code name assigned in {\normalfont KONECT}. $n$ shows the vertex size of the (largest) connected network. %
    The column $nnz$ denotes the number of nonzero elements of its Laplacian matrix and the last column records the time costs of the three methods.}\label{tab:konect_sira}
\end{table}
\begin{figure}[!ht]
\centering
 \begin{tabular}{cc}
 \subfigure[Sparse Pattern of {\bf FX}]{\label{fig:fx_kmr}\includegraphics[width=0.46\textwidth]{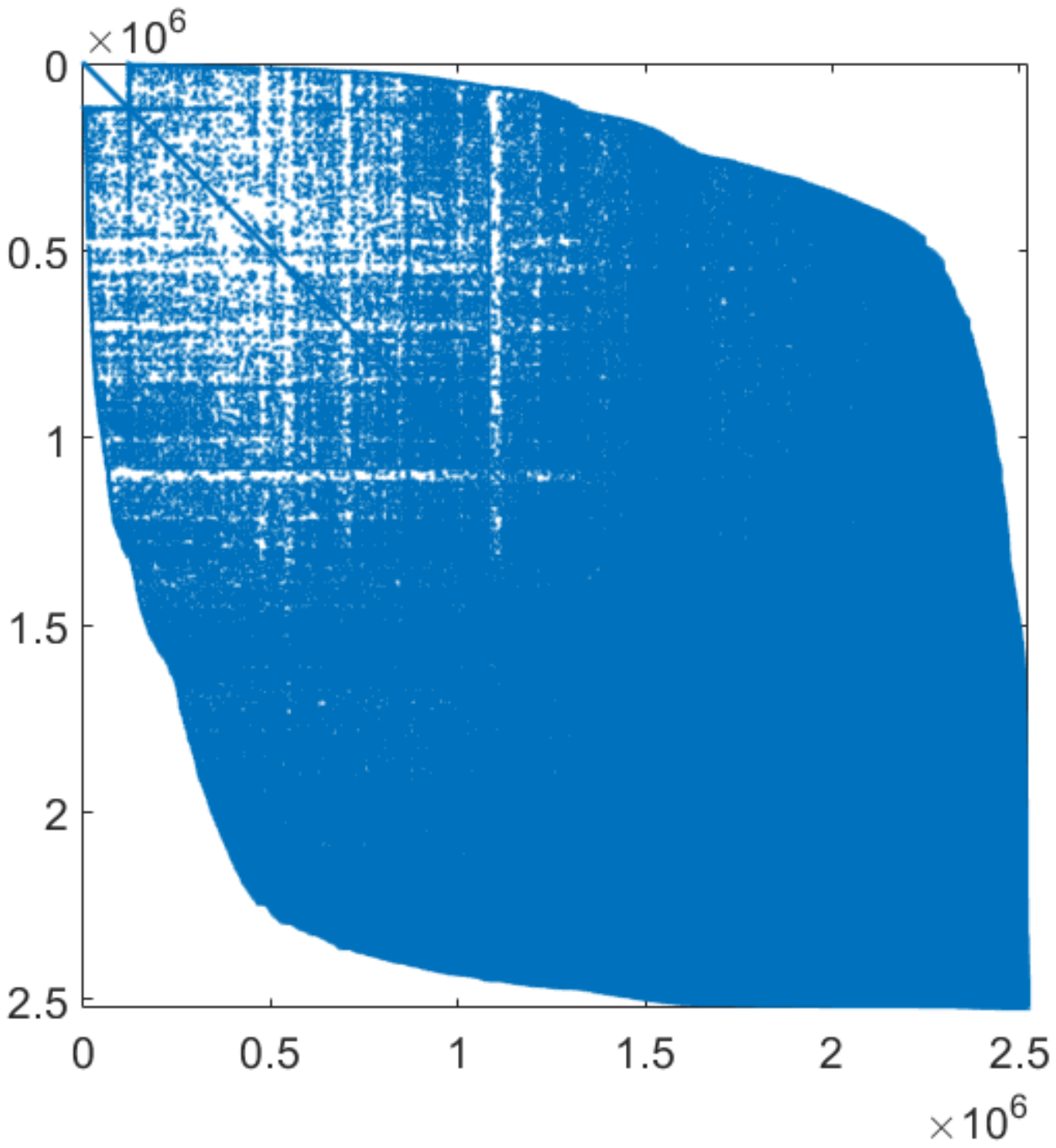}}
 \subfigure[Residual Evolution of {\bf FX}]{\label{fig:yt_tdsira}\includegraphics[width=0.48\textwidth]{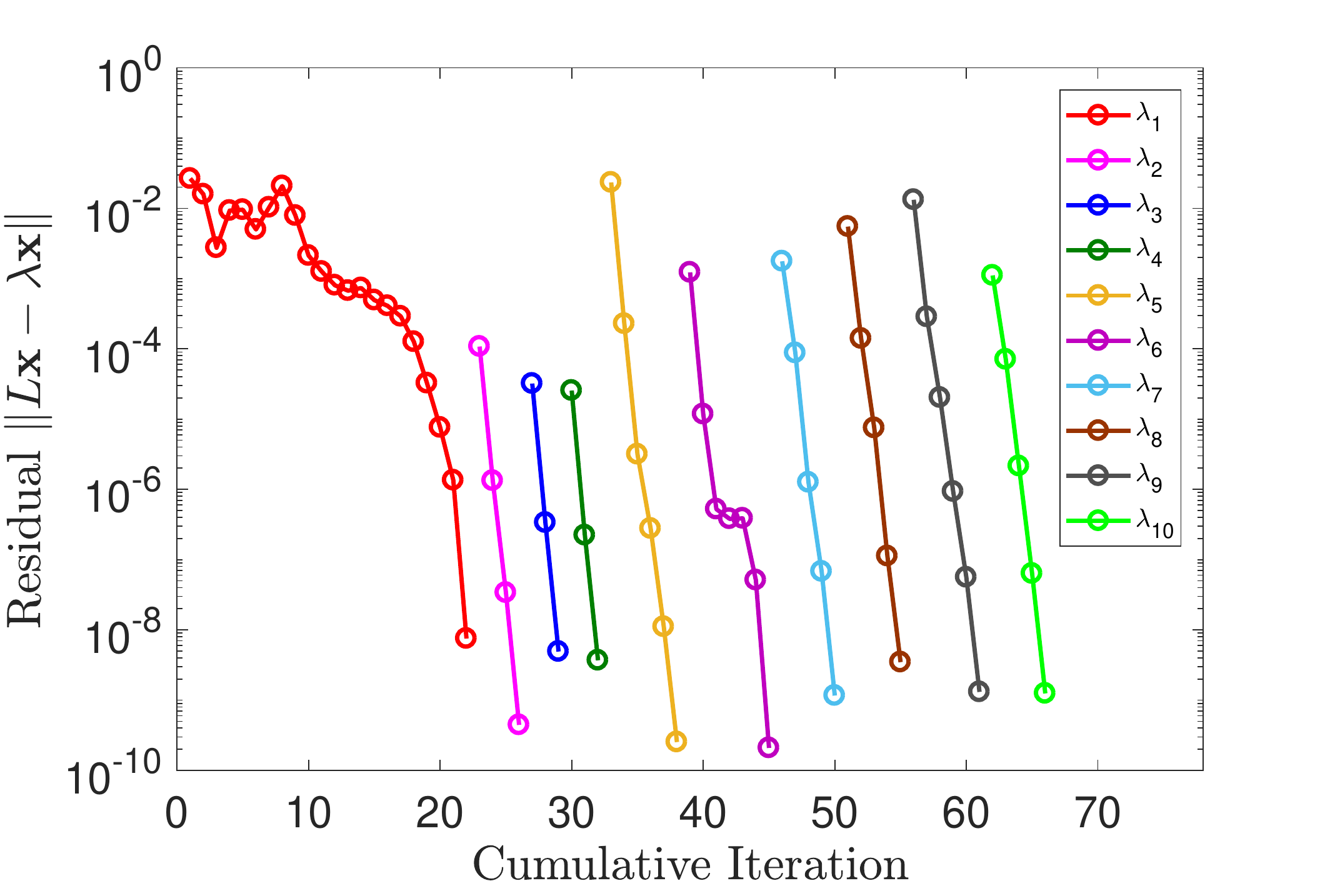}}\\
 \subfigure[Sparse Pattern of {\bf YT}]{\label{fig:fx_kmr}\includegraphics[width=0.46\textwidth]{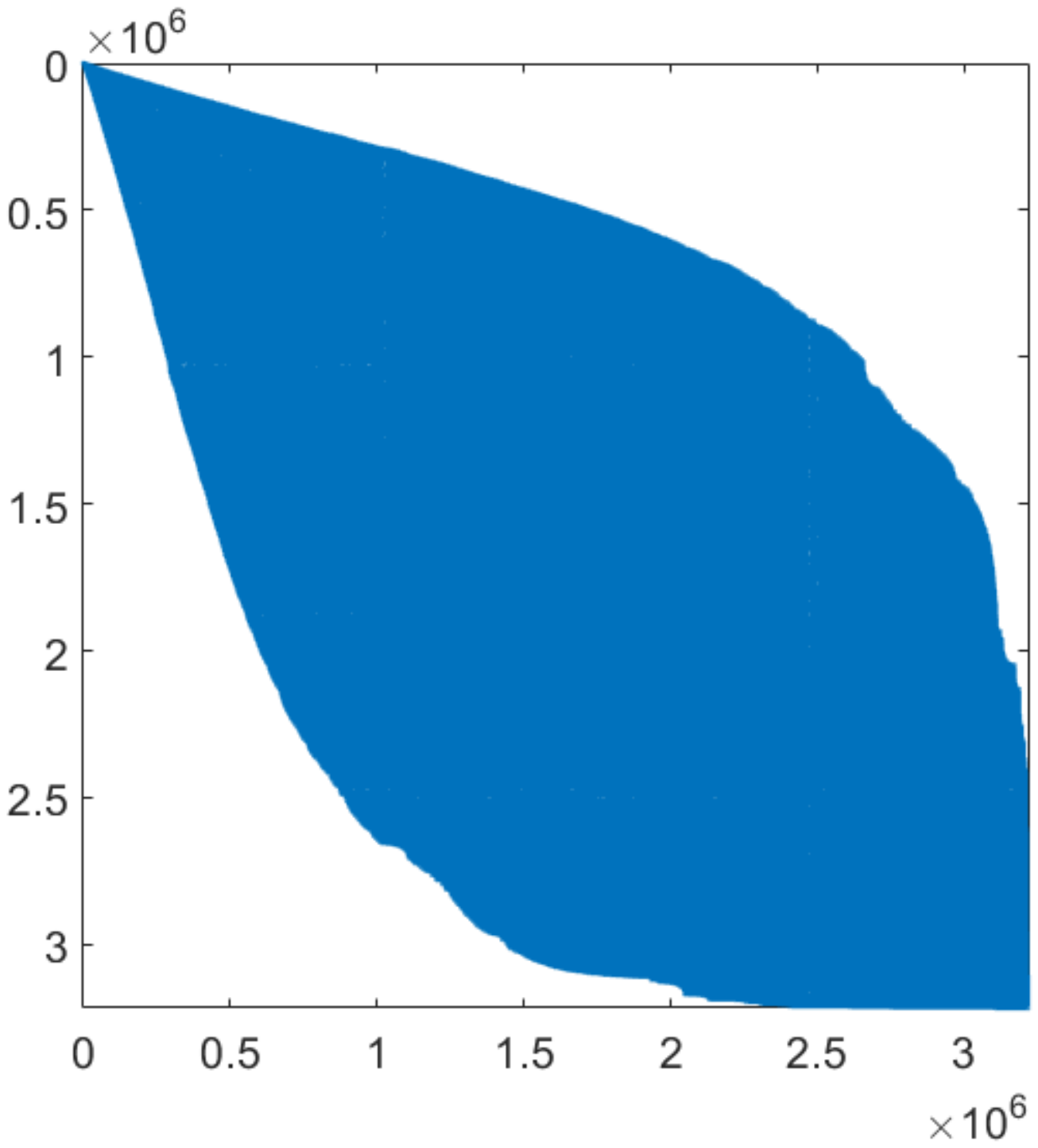}}
 \subfigure[Residual Evolution of {\bf YT}]{\label{fig:yt_tdsira}\includegraphics[width=0.48\textwidth]{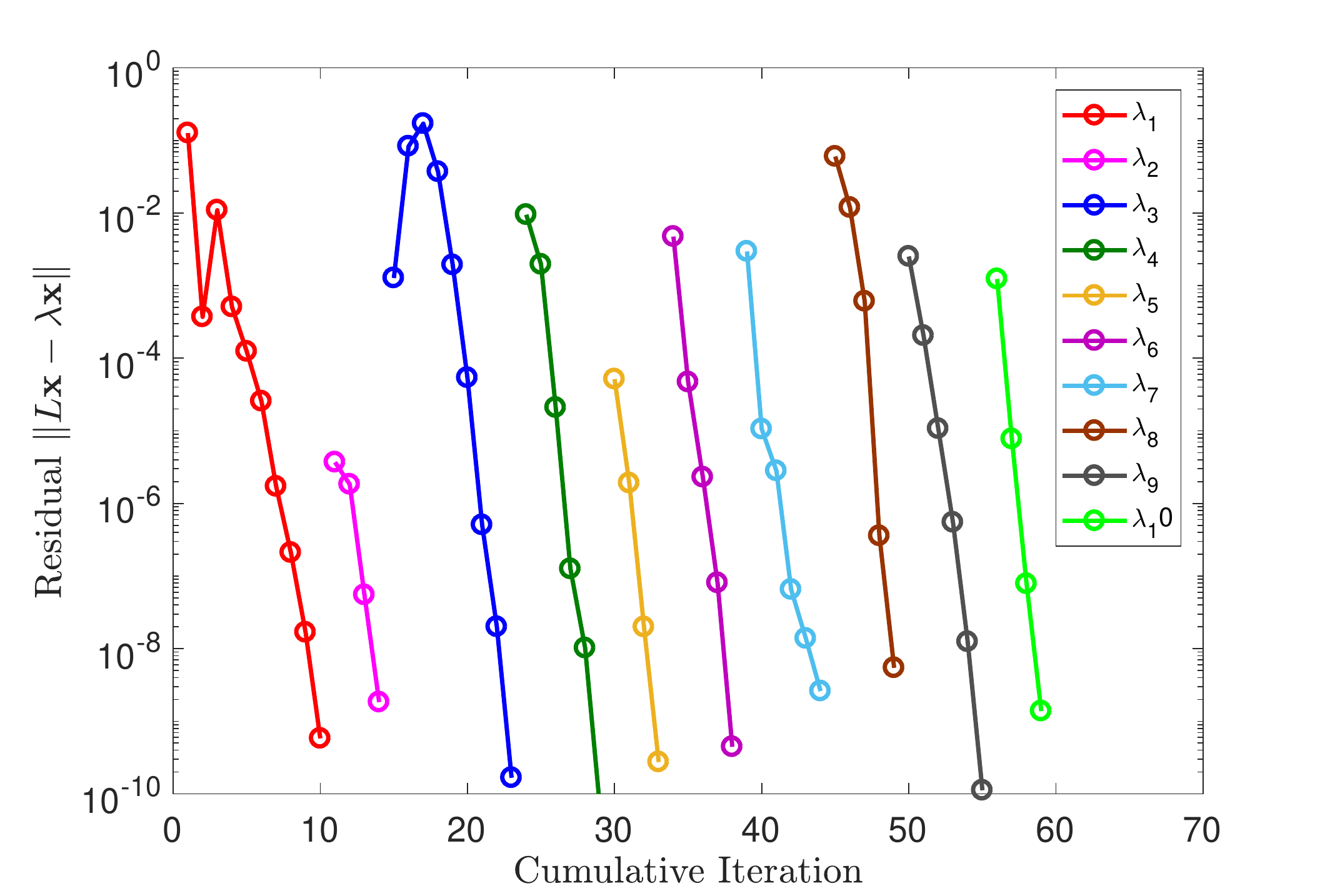}}\\
 \end{tabular}
 \caption{Sparsity patterns of the reverse Cuthill-McKee ordering and convergence histories of the {\tt isira} method for the networks {\bf FX} and {\bf YT}.}
     \label{fig:pattern_residual}
\end{figure}

For large networks whose bandwidth are very wide after performing the reordering %
such the minimum degree algorithm \cite{M57}, Cuthill-McKee algorithm \cite{CK69} and the reverse Cuthill-McKee algorithm\cite{G76}, %
to achieve the matrix factorizations will suffer from the problem on out of memory. %
Figure~\ref{fig:pattern_residual} shows matrix pattern of {\bf FX} and {\bf YT} after performing the reverse Cuthill-McKee algorithm. %
As we can see that the sparse pattern is still very \textquotedblleft fat\textquotedblright and this will make the factorization produce dense matrices. %
In such situations, the {\tt isira} method circumvent this difficulty %
using the {\it iterative method} and {\it inexactly solving} the linear systems contained the eigensolver. %

Table~\ref{tab:konect_sira} reveals feasibility of {\tt isira} to capture some smallest eigenvalues %
in a reasonable time cost. All networks have similar convergence processes and %
we present iteration behaviors of {\bf FX} and {\bf YT} in Figure~\ref{fig:pattern_residual}. %
Moreover, in Table~\ref{tab:konect_sira}, %
we also see that the techniques from our previous method {\tt ndeigs} \cite{HGLY14} can adjust the structure of a matrix without changing its sparsity of the matrix. %
This can reduce possibility of zero diagonal elements appearing in the execution of the matrix factorization. %
We will demonstrate this effect again in Remark~\ref{rmk:chol_band}.%

\begin{example}\label{exp:trimID}
In Remark~\ref{rmk:trimID}, we make some interpretable idea to select the vertex whose degree is the maximum as the trimming index. %
Table~\ref{tab:trim_index} shows some comparisons between the vertex degree and the computational time. %
\begin{table}[!ht]
\centering{\scriptsize
    \begin{tabular}{cc||cc||cc||cc||cc}
    \toprule
    \multicolumn{2}{c||}{\bf CY} & \multicolumn{2}{c||}{\bf HY} & \multicolumn{2}{c||}{\bf FX} & \multicolumn{2}{c||}{\bf YT} & \multicolumn{2}{c}{\bf SK}\\ \hline\hline
      degree &         time  & degree   &  time         & degree  &  time          & degree & time          & degree & time         \\ \hline
           1 &         406.6 &      1   &  319.2        &      1  &  784.5         &      1 & 1462.2        &      1 & 820.1        \\ \hline
          29 &         360.0 &    172   &  262.0        &     57  &  744.5         &     29 & 1379.3        &    360 & 747.6        \\ \hline
      28,754 &         315.3 & 31,883   &  223.3        &  1,474  &  740.4         & 91,751 & 1068.0        & 35,455 & 660.0        \\
    \bottomrule
    \end{tabular}}
\caption{The CPU time of the {\tt isira} method using different trimming indices.}\label{tab:trim_index}
\end{table}
\end{example}

In the first four examples, the third row is exact the index $n$, the fourth row presents the results by deleting the first row and column, and the last row is obtained by trimming the vertex which has the maximum degree. %
In the case of {\bf SK}, the minimum degree is still equal to one but is not equal to the last index. %
The last row also picks the index with the maximum degree. %
The value 360 in the network {\bf SK} is the degree of its first vertex. %
On our experimental experience, trimming the vertex with the smallest degree always get slower computational efficiency. %
This may indirectly display that trimming the index of the vertex with the maximum degree may be an appropriate choice.

\begin{figure}[!ht]
\centering
 \begin{tabular}{cc}
 \subfigure[{\bf WO}]{\label{fig:wo_kmr}\includegraphics[width=0.185\textwidth]{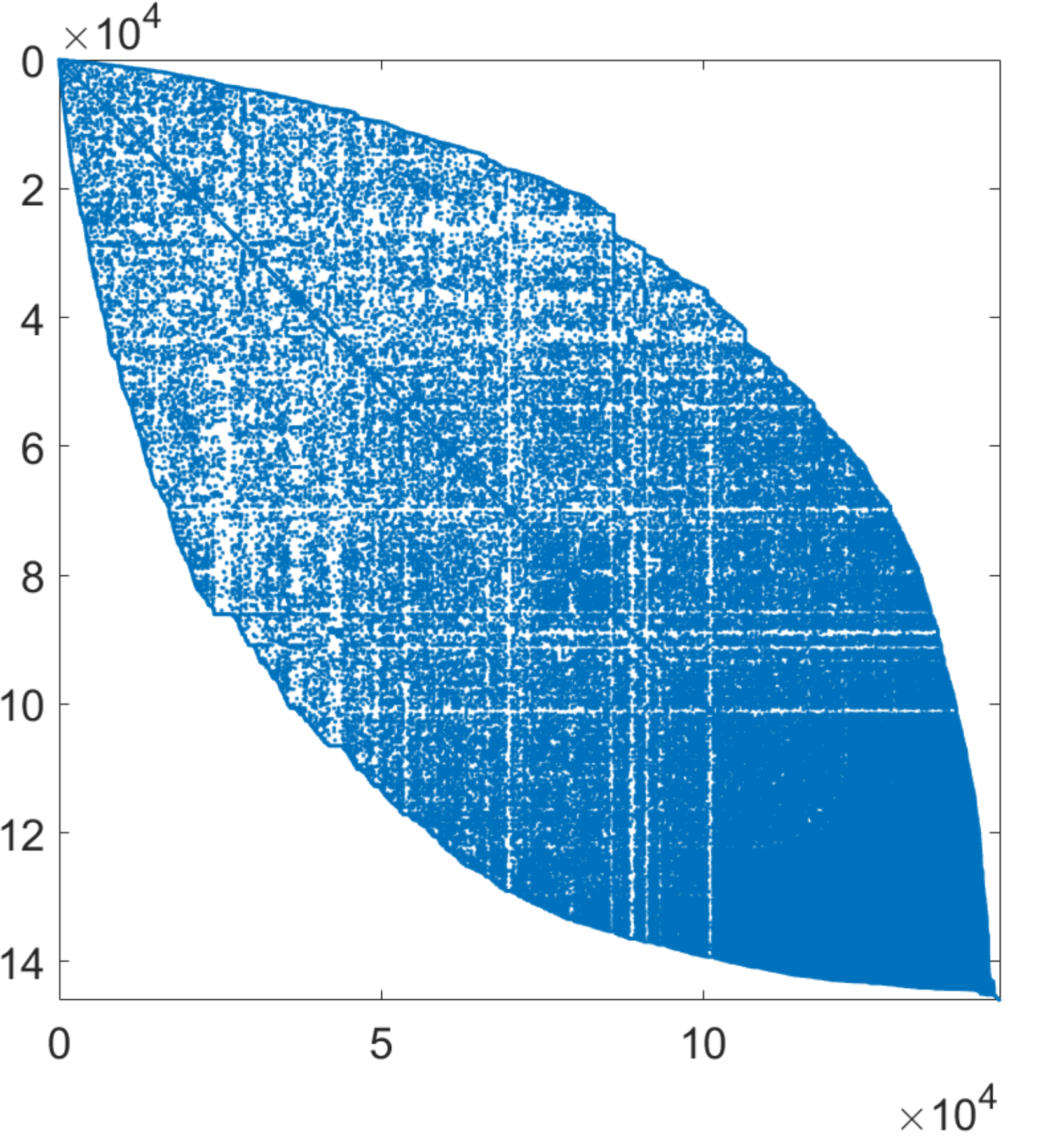}}
 \subfigure[{\bf CD}]{\label{fig:cd_kmr}\includegraphics[width=0.185\textwidth]{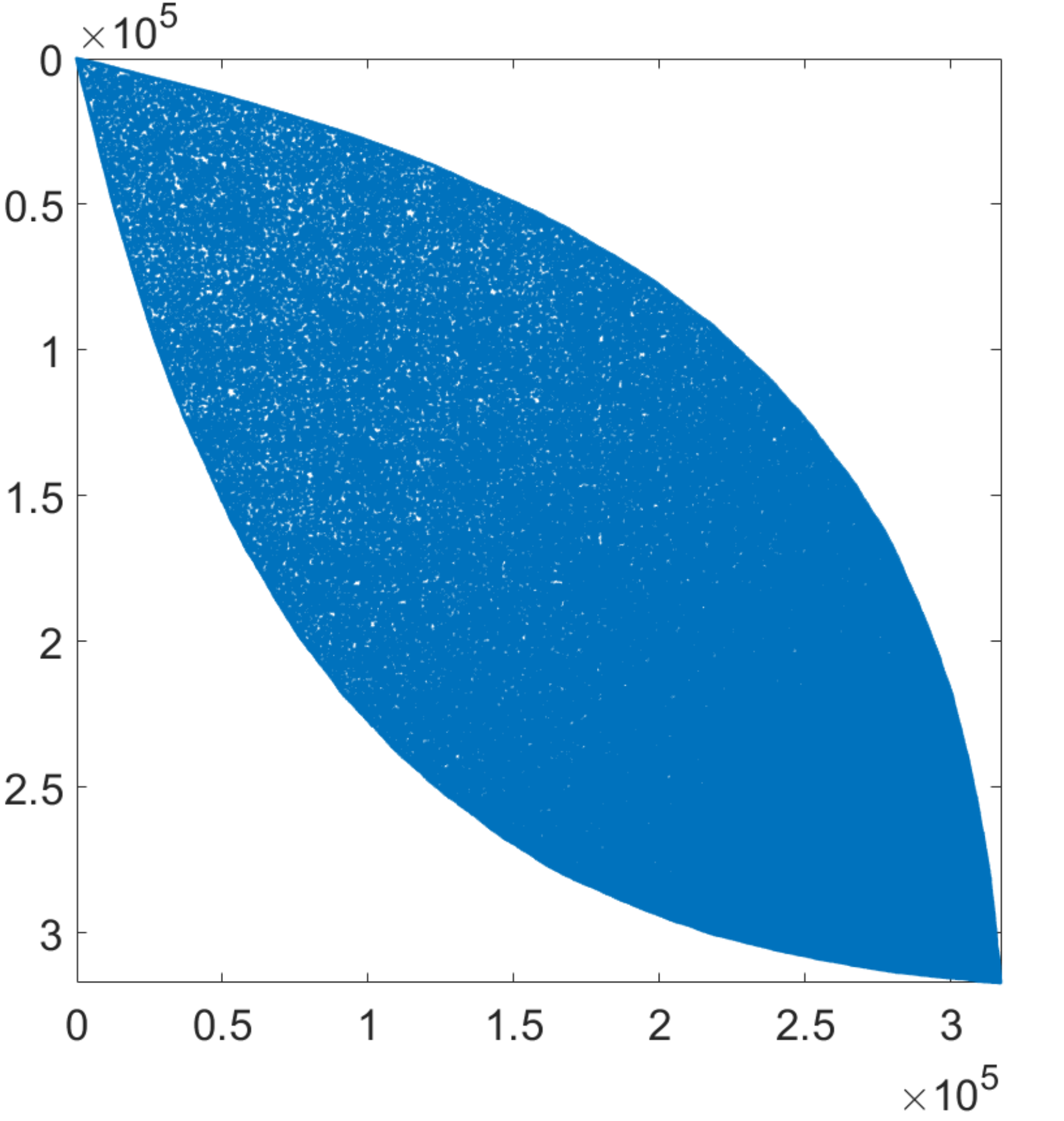}}
 \subfigure[{\bf RD}]{\label{fig:rd_kmr}\includegraphics[width=0.18\textwidth]{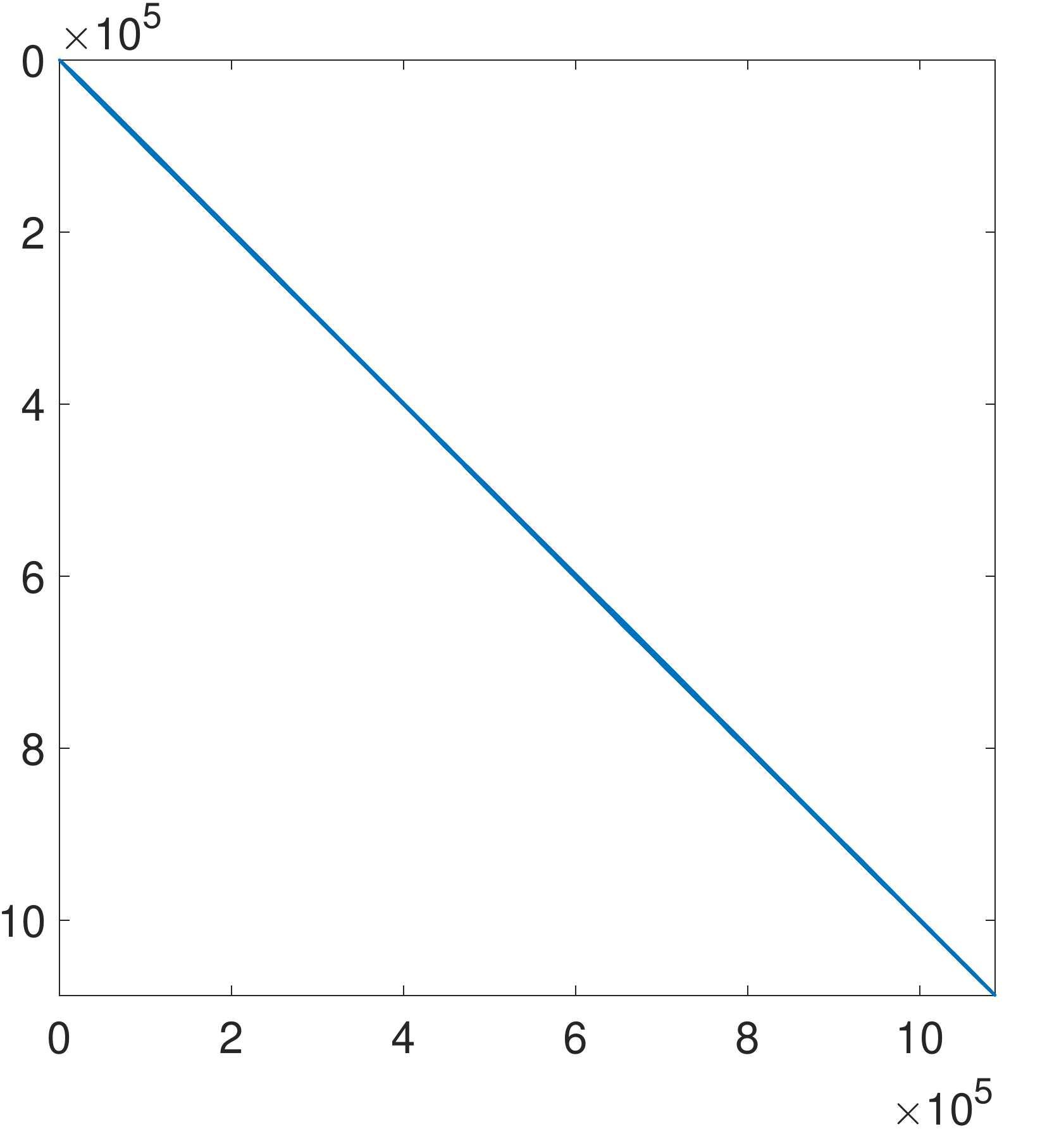}}
 \subfigure[{\bf R1}]{\label{fig:r1_kmr}\includegraphics[width=0.18\textwidth]{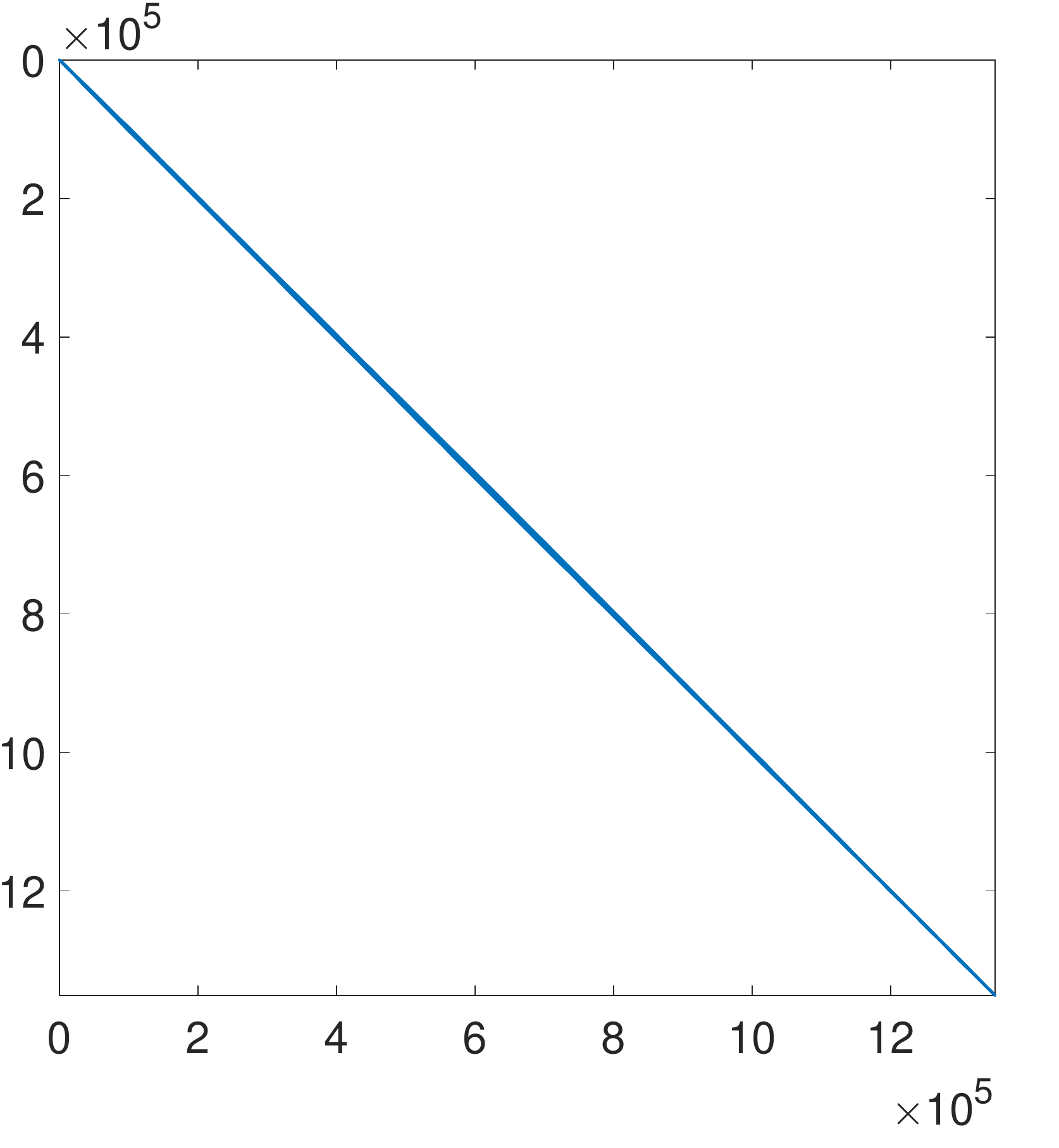}}
 \subfigure[{\bf RO}]{\label{fig:ro_kmr}\includegraphics[width=0.18\textwidth]{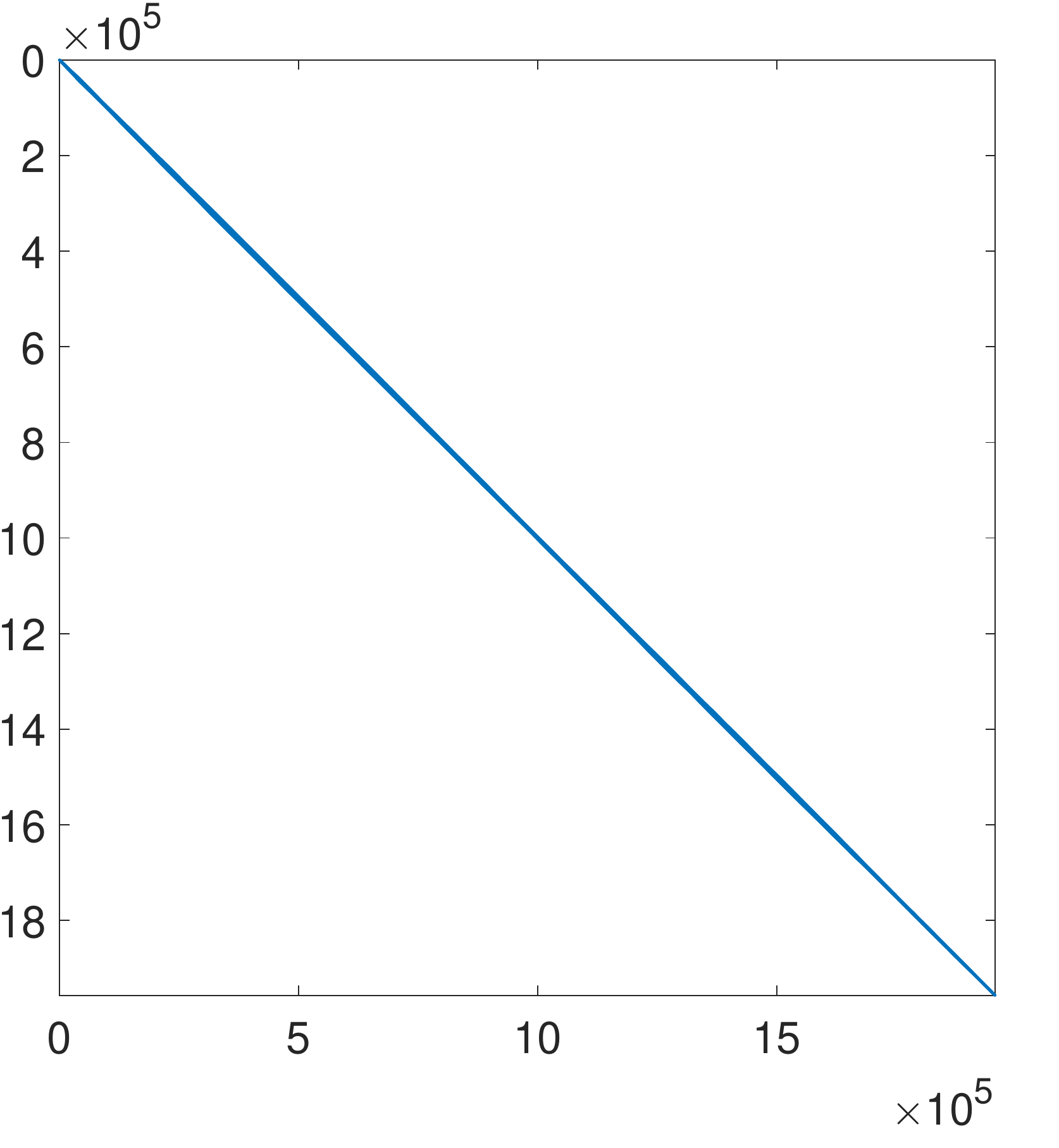}}
 \end{tabular}
 \caption{Networks in Table~\ref{tab:time_chol} have narrower bandwidth patterns after reordering.}\label{fig:pattern_band}
\end{figure}
\begin{table}[!ht]
\centering
\begin{tabular}{c|c|c|D{.}{.}{2}D{.}{.}{2}}
\multirow{2}{*}{Code} & \multirow{2}{*}{$n$} & \multirow{2}{*}{$nnz$} & \multicolumn{2}{c}{Time (sec)} \\ \cline{4-5}
                      & & & \multicolumn{1}{c}{\tt ndeigs} & \multicolumn{1}{c}{\tt eigs}\\
\hline
{\bf WO}  &\multicolumn{1}{r|}{  145,145} & \multicolumn{1}{r|}{ 1,457,605}  & 15.3 &  31.2 \\
{\bf CD}  & \multicolumn{1}{r|}{  317,080} & \multicolumn{1}{r|}{2,416,812}  & 76.6 & 249.5 \\
{\bf RD}  &\multicolumn{1}{r|}{1,087,562} & \multicolumn{1}{r|}{ 4,170,590}  &  8.0 &   9.6 \\
{\bf R1}  &\multicolumn{1}{r|}{1,351,137} & \multicolumn{1}{r|}{ 5,109,539}  &  8.0 &  11.7 \\
{\bf RO}  &\multicolumn{1}{r|}{1,957,027} & \multicolumn{1}{r|}{ 7,477,803}  & 11.9 &  16.9 \\
\end{tabular}
\caption{The CPU time of the {\tt eigs} and {\tt ndeigs} for networks having {\it thinner bandwidth} after the index permutation as shown in Figure~\ref{fig:pattern_band}. The contents of the first row of this table are as introduced in Table~\ref{tab:konect_sira}.}\label{tab:time_chol}
\end{table}

\begin{remark}\label{rmk:chol_band}
The {\tt isira} is committed to solve the {\normalfont GLEP} whose network pattern is wide even after suitable reordering of vertices. %
That is, in this case, the graph Laplacian matrix has no {\normalfont LU} factorization available. %
On the contrary, if the bandwidth of a matrix becomes narrow after applying some reordering algorithms, %
then the {\normalfont LU} factorization is very \textquotedblleft cheap\textquotedblright\ either on the demand of memory or the time of computing. %
Under such situations, {\tt eigs} and {\tt ndeigs} become preferable %
since the inner linear systems can be solved by the {\normalfont LU} factorization. %

Figure~\ref{tab:time_chol} gives five examples that the matrices have narrower bandwidth patterns %
after performing the reverse {\normalfont Cuthill-McKee} algorithm (the {\normalfont MATLAB} built-in function: {\tt symrcm}). %
Tale~\ref{fig:pattern_band} presents the information of these networks and the time cost of {\tt eigs} and {\tt ndeigs} methods. %
It is worth to mention that the rank-two correction strategy in {\tt ndeigs} \cite{HGLY14}, our previous method, makes the performance be superior to {\tt eigs}. %
\end{remark}

\section{Conclusions and Future Works}\label{sec:Conclusion}
Graph Laplacian eigenvalue problems focus on the extraction of potential quantities, such as eigenvalues and eigenvectors of network structures. %
Due to the graph Laplacian matrix, generated from a simple, connected and undirected graph, is symmetric positive semi-definite and the information what we are interested often be hidden in the eigenvalues with small magnitude. %
This indicates that we have to address the problem of solving a singular system. %
Thus, the study of efficient eigensolvers for finding some smallest positive eigenvalues %
and associated eigenvectors of the graph Laplacian eigenvalue problem is a challenging and important topic. %

To this end, based on the method of inexact shift-invert residual Arnoldi (SIRA) \cite{L07,LS07}, %
we derive a trimming technique in Section~\ref{sec:trimming} to %
implicitly remedy the singularity of the graph Laplacian matrix and get a solution %
which is orthogonal to the null space for expanding the searching subspace. %
Furthermore, we apply the deflating approach to exclude the influence of convergence eigenvalues %
so as to focus on capturing the smallest positive one and detect a few desired eigenvalues in order.
Numerical experiments show that, compared with the numerical implementations of traditional diagonal perturbation {\tt eigs} and the null-space deflation eigensolver {\tt ndeigs} \cite{HGLY14} %
using direct methods for inner linear systems, %
the new derived algorithm iSIRA, integrated SIRA (SIRA with trimming and deflating techniques), %
reveals more efficient and feasible when a LU factorization of the graph Laplacian matrix is not available.

According to the concise algorithm of SIRA, in the future, %
we will focus on the development a GPU version of the iSIRA algorithm (Algorithm~\ref{alg:SIRAtd}). %
In addition, we will pay attention to Laplacian matrix arising from a direct graph. %
Note that in such a case, all properties and advantages of a undirect Laplacian proposed in this paper are no longer be applied. %
How to explore new features and develop a fast eigensolver is under investigation.

\section{Acknowledgments}
The first three authors would like to thank the grant support from the Ministry of Science and Technology in Taiwan, and %
the ST Yau Center and Big Data Research at the National Chiao Tung University. %
The first and the third authors would further thank Dr. Pin-Yu Chen at IBM’s T.J. Watson Research Center %
for proving application information about this topic.
\bibliographystyle{abbrv}
\bibliography{references}

\begin{thebibliography}{10}

\bibitem{ADGKMZ08}
A.~Arenas, A.~D'az-Guilera, J.~Kurths, Y.~Moreno, and C.~Zhou.
\newblock Synchronization in complex networks.
\newblock {\em Phys. Rep.}, 469(3):93--153, 2008.

\bibitem{B_00}
Z.~Bai, J.~Demmel, J.~Dongarra, A.~Ruhe, and H.~van~der Vorst.
\newblock {\em Templates for the Solution of Algebraic Eigenvalue Problems: A
  Practical Guide}.
\newblock SIAM: Society for Industrial and Applied Mathematics, Philadelphia,
  PA, 2000.

\bibitem{BN03}
M.~Belkin and P.~Niyogi.
\newblock Laplacian eigenmaps for dimensionality reduction and data
  representation.
\newblock {\em Neural Comput.}, 15(6):1373--1396, 2003.

\bibitem{CH15A}
P.~Y. Chen and A.~O. Hero.
\newblock Deep community detection.
\newblock {\em IEEE Tr. Signal Proces.}, 63(21):5706--5719, Nov 2015.

\bibitem{CH15IP}
P.~Y. Chen and A.~O. Hero.
\newblock Phase transitions in spectral community detection of large noisy
  networks.
\newblock In {\em 2015 IEEE International Conference on Acoustics, Speech and
  Signal Processing (ICASSP)}, pages 3402--3406, April 2015.

\bibitem{CK69}
E.~Cuthill and J.~McKee.
\newblock Reducing the bandwidth of sparse symmetric matrices.
\newblock In {\em Proceedings of the 1969 24th National Conference}, ACM '69,
  pages 157--172, New York, NY, USA, 1969. ACM.

\bibitem{F10}
S.~Fortunato.
\newblock Community detection in graphs.
\newblock {\em Phys. Rep.}, 486(3):75--174, 2010.

\bibitem{G76}
N.~E. Gibbs.
\newblock Algorithm 509: A hybrid profile reduction algorithm [{F}1].
\newblock {\em ACM Trans. Math. Softw.}, 2(4):378--387, Dec. 1976.

\bibitem{GvL12}
G.~H. Golub and C.~F. Van~Loan.
\newblock {\em Matrix Computations}.
\newblock The Johns Hopkins University Press, Baltimore, MD, 4th edition, 2012.

\bibitem{HHLW14}
T.-M. Huang, H.-E. Hsieh, W.-W. Lin, and W.~Wang.
\newblock Eigenvalue solvers for three dimensional photonic crystals with
  face-centered cubic lattice.
\newblock {\em J. Comput. Appl. Math.}, 272:350 -- 361, 2014.

\bibitem{HGLY14}
W.-Q. Huang, X.~D. Gu, W.-W. Lin, and S.-T. Yau.
\newblock A novel symmetric skew-{H}amiltonian isotropic {L}anczos algorithm
  for spectral conformal parameterizations.
\newblock {\em J Sci. Comput.}, 61(3):558--583, Dec 2014.

\bibitem{I97}
I.~C.~F. Ipsen.
\newblock Computing an eigenvector with inverse iteration.
\newblock {\em SIAM Rev.}, 39(2):254--291, 1997.

\bibitem{JL14}
Z.~Jia and C.~Li.
\newblock Inner iterations in the shift-invert residual {A}rnoldi method and
  the {J}acobi-{D}avidson method.
\newblock {\em Sci. China Math.}, 57(8):1733--1752, 2014.

\bibitem{KOSZ13}
J.~A. Kelner, L.~Orecchia, A.~Sidford, and Z.~A. Zhu.
\newblock A simple, combinatorial algorithm for solving sdd systems in
  nearly-linear time.
\newblock In {\em Proceedings of the Forty-fifth Annual ACM Symposium on Theory
  of Computing}, STOC '13, pages 911--920, New York, NY, USA, 2013. ACM.

\bibitem{KMP11}
I.~Koutis, G.~L. Miller, and R.~Peng.
\newblock A nearly-m log n time solver for {SDD} linear systems.
\newblock In {\em 2011 IEEE 52nd Annual Symposium on Foundations of Computer
  Science}, pages 590--598, Oct 2011.

\bibitem{KMT11}
I.~Koutis, G.~L. Miller, and D.~Tolliver.
\newblock Combinatorial preconditioners and multilevel solvers for problems in
  computer vision and image processing.
\newblock {\em Comput. Vis. Image. Und.}, 115(12):1638--1646, 2011.
\newblock Special issue on Optimization for Vision, Graphics and Medical
  Imaging: Theory and Applications.

\bibitem{K13}
J.~Kunegis.
\newblock Konect: The {K}oblenz {N}etwork {C}ollection.
\newblock In {\em Proceedings of the 22nd International Conference on World
  Wide Web}, WWW '13 Companion, pages 1343--1350, New York, NY, USA, 2013. ACM.

\bibitem{L07}
C.-R. Lee.
\newblock {\em Residual {A}rnoldi method: theory, package and experiments}.
\newblock PhD thesis, Department of Computer Science,University of Maryland at
  College Park, 2007.

\bibitem{LS07}
C.-R. Lee and G.~W. Stewart.
\newblock Analysis of the residual {A}rnoldi method.
\newblock Technical report, Department of Computer Science,University of
  Maryland at College Park, 2007.

\bibitem{M57}
H.~M. Markowitz.
\newblock The elimination form of the inverse and its application to linear
  programming.
\newblock {\em Manage. Sci.}, 3(3):255--269, 1957.

\bibitem{M97}
B.~Mohar.
\newblock {\em Some applications of {L}aplace eigenvalues of graphs}, pages
  225--275.
\newblock Springer Netherlands, Dordrecht, 1997.

\bibitem{MP93}
B.~Mohar and S.~Poljak.
\newblock {\em Eigenvalues in Combinatorial Optimization}, pages 107--151.
\newblock Springer New York, New York, NY, 1993.

\bibitem{M12}
J.~J. Molitierno.
\newblock {\em Applications of combinatorial matrix theory to {L}aplacian
  matrices of graphs}.
\newblock Discrete Mathematics and Its Applications. CRC Press, Boca Raton,
  Florida, 2012.

\bibitem{N06}
M.~E.~J. Newman.
\newblock Finding community structure in networks using the eigenvectors of
  matrices.
\newblock {\em Phys. Rev. E}, 74:036104, Sep 2006.

\bibitem{N04}
A.~Y. Ng.
\newblock Feature selection, {L}1 vs. {L}2 regularization, and rotational
  invariance.
\newblock In {\em Proceedings of the Twenty-first International Conference on
  Machine Learning}, ICML '04, pages 78--, New York, NY, USA, 2004. ACM.

\bibitem{NJY02}
A.~Y. Ng, M.~I. Jordan, and Y.~Weiss.
\newblock On spectral clustering: Analysis and an algorithm.
\newblock In T.~G. Dietterich, S.~Becker, and Z.~Ghahramani, editors, {\em
  Advances in Neural Information Processing Systems 14}, pages 849--856. MIT
  Press, 2002.

\bibitem{PCOA15}
J.~Pang, G.~Cheung, A.~Ortega, and O.~C. Au.
\newblock Optimal graph laplacian regularization for natural image denoising.
\newblock In {\em 2015 IEEE International Conference on Acoustics, Speech and
  Signal Processing (ICASSP)}, pages 2294--2298, April 2015.

\bibitem{P98}
B.~N. Parlett.
\newblock {\em The Symmetric Eigenvalue Problem}.
\newblock SIAM: Society for Industrial and Applied Mathematics, 1998.

\bibitem{Q09}
Z.~Qu.
\newblock {\em Cooperative Control of Dynamical Systems: Applications to
  Autonomous Vehicles}.
\newblock Springer-Verlag, 2009.

\bibitem{S07}
S.~E. Schaeffer.
\newblock Graph clustering.
\newblock {\em Comput. Sci. Rev.}, 1(1):27--64, 2007.

\bibitem{SM00}
J.~Shi and J.~Malik.
\newblock Normalized cuts and image segmentation.
\newblock {\em IEEE T. Pattern Anal.}, 22(8):888--905, Aug 2000.

\bibitem{ST04}
D.~A. Spielman and S.-H. Teng.
\newblock Nearly-linear time algorithms for graph partitioning, graph
  sparsification, and solving linear systems.
\newblock In {\em Proceeding {STOC }'04 {P}roceedings of the thirty-sixth
  annual {ACM} symposium on {T}heory of computing}, pages 81--90, New York, NY,
  USA, 2004. ACM.

\bibitem{ST14}
D.~A. Spielman and S.-H. Teng.
\newblock Nearly linear time algorithms for preconditioning and solving
  symmetric, diagonally dominant linear systems.
\newblock {\em SIAM J. Matrix. Anal. A.}, 35(3):835--885, 2014.

\bibitem{S01}
G.~W. Stewart.
\newblock {\em Matrix Algorithms: Volume II: Eigensystems}.
\newblock SIAM: Society for Industrial and Applied Mathematics, 2001.

\bibitem{TA77}
A.~N. Tikhonov and V.~t. V.~Y.~Arsenin.
\newblock {\em Solutions of ill posed problems}.
\newblock V. H. Winston \& Sons, Washington, D.C.: Winston, 1977.

\bibitem{V91}
P.~M. Vaidya.
\newblock Solving linear equations with symmetric diagonally dominant matrices
  by constructing good preconditioners.
\newblock {\em A talk based on this manuscript}, 2(3.4):2--4, 1991.

\bibitem{vL07}
U.~von Luxburg.
\newblock A tutorial on spectral clustering.
\newblock {\em Stat. Comput.}, 17(4):395--416, Dec 2007.

\bibitem{WS05}
S.~White and P.~Smyth.
\newblock {\em A Spectral Clustering Approach to Finding Communities in
  Graphs}, pages 274--285.
\newblock Proceedings of the 2005 SIAM International Conference on Data Mining.
  2005.

\bibitem{XW05}
R.~Xu and D.~Wunsch.
\newblock Survey of clustering algorithms.
\newblock {\em IEEE T. Neural Networ.}, 16(3):645--678, May 2005.

\end{thebibliography}

\end{document}